\documentclass[reqno]{amsart}

\usepackage[T1]{fontenc}
\usepackage[utf8]{inputenc}
\usepackage[english]{babel}

\usepackage[foot]{amsaddr}

\usepackage{mathtools,amssymb,amsthm}
\usepackage{enumitem}
\usepackage{caption,subcaption}
\newsavebox{\largestimage}

\usepackage{multirow}
\usepackage[x11names]{xcolor}
\usepackage{tikz}
\usetikzlibrary{decorations.pathreplacing}
\tikzstyle{every picture}=[line join=round,line cap=round,every label/.append style={font=\small},label distance=-1pt,line width=.8pt]

\setenumerate{label={{\rm({\roman*})}},leftmargin=8mm,itemsep=3pt,topsep=3pt}
\setitemize{label={$\vcenter{\hbox{\tiny$\bullet$}}$},left margin=8mm,itemsep=3pt,topsep=3pt}

\usepackage[a4paper,tmargin=2.5cm,rmargin=2cm]{geometry}

\xdefinecolor{linkblue}{RGB}{20,65,109}
\xdefinecolor{lightblue}{RGB}{66,133,244}
\xdefinecolor{darkblue}{RGB}{0,0,139}
\xdefinecolor{DeepPink}{RGB}{255, 20, 147}

\usepackage[colorlinks,linkcolor=linkblue,cite color=linkblue,hypertexnames=true,pdftex,bookmarks=true,bookmarksnumbered=true]{hyperref}

\usepackage[nice]{nicefrac}

\makeatletter
\renewcommand{\email}[2][]{%
  \ifx\emails\@empty\relax\else{\g@addto@macro\emails{,\space}}\fi%
  \@ifnotempty{#1}{\g@addto@macro\emails{\textrm{(#1)}\space}}%
  \g@addto@macro\emails{#2}%
}
\makeatother

\makeatletter
\newtheorem*{rep@theorem}{\rep@title}
\newcommand{\newreptheorem}[2]{%
\newenvironment{rep#1}[1]{%
 \def\rep@title{#2 \ref{##1}}%
 \begin{rep@theorem}}%
 {\end{rep@theorem}}}
\makeatother

\newreptheorem{theorem}{Theorem}
\newreptheorem{corollary}{Corollary}

\newtheorem{thm}{Theorem}

\newtheorem{clm}[thm]{Claim}
\newtheorem{lem}[thm]{Lemma}
\theoremstyle{remark}

\newenvironment{subproof}[1][\proofname]{%

\begin{proof}[#1]}{\end{proof}}

\renewcommand{\epsilon}{\varepsilon}
\renewcommand{\le}{\leqslant}
\renewcommand{\ge}{\geqslant}
\renewcommand{\leq}{\leqslant}
\renewcommand{\geq}{\geqslant}

\newcommand{\bc}{\mathsf{bc}}

\newcommand{\out}{\mathsf{Out}}

\makeatletter
\renewcommand{\email}[2][]{%
  \ifx\emails\@empty\relax\else{\g@addto@macro\emails{,\space}}\fi%
  \@ifnotempty{#1}{\g@addto@macro\emails{\textrm{(#1)}\space}}%
  \g@addto@macro\emails{#2}%
}
\g@addto@macro\normalsize{
\setlength{\abovedisplayskip}{1.5mm}
\setlength{\belowdisplayskip}{1.5mm}
\setlength{\abovedisplayshortskip}{1.5mm}
\setlength{\belowdisplayshortskip}{1.5mm}}
\makeatother

\begin{document}

\title{The biclique covering number of grids}
\author[K.~Guo]{Krystal Guo\textsuperscript{1}}
\email[1]{guo.krystal@gmail.com}
\author[T.~Huynh]{Tony Huynh\textsuperscript{2}}
\email[2]{tony.bourbaki@gmail.com}
\author[M.~Macchia]{Marco Macchia\textsuperscript{3}}
\email[3]{mmacchia@ulb.ac.be}
\address[1,2,3]{D\'epartement de Math\'ematique, Universit\'e libre de Bruxelles}
\date{}

\begin{abstract}
We determine the exact value of the biclique covering number for all grid graphs. 
\end{abstract}

\maketitle

\section{Introduction} \label{sec:introduction}
Let $G$ be a graph.  A \emph{biclique} of $G$ is a complete bipartite subgraph.
The \emph{biclique covering number} of $G$, denoted $\bc(G)$, is the minimum number of bicliques of $G$ required to cover the edges of $G$.
The biclique covering number is studied in fields as diverse as polyhedral combinatorics~\cite{KW15, AFFHM17}, biology~\cite{NMW78}, and communication complexity~\cite{razborov90}, where it is also known as \emph{bipartite dimension} or \emph{rectangle covering number}.  

Computing the biclique covering number is a classic $\mathsf{NP}$-hard problem. Indeed, deciding if $\bc(G) \le k$ appears as problem \textbf{GT18} in Garey and Johnson~\cite{GJ79}.  It is also $\mathsf{NP}$-hard to approximate within a factor of $n^{1-\epsilon}$, where $n$ is the number of vertices of $G$ (see~\cite{CHHK14}).  

As such, there are very few classes of graphs for which we know the biclique covering number exactly. For example, it is well-known that the biclique covering number of the complete graph $K_n$ is $\lceil \log_2 n \rceil$ (see~\cite{FH96} for a proof).  Note that the minimum number of bicliques that \emph{partition} $E(K_n)$ is $n-1$ by the Graham-Pollak theorem~\cite{GP71}. 
This result was later extended to biclique coverings $\mathcal{C}$ of $K_n$ such that each edge is in at most $k$ bicliques of $\mathcal C$.  Alon~\cite{Alon97} showed that the minimum number of bicliques in such coverings is $\Theta(kn^{\nicefrac{1}{k}})$.
Two more classes of graphs for which we know the  biclique covering number exactly are $K_{2n}^-$ and $K_{n,n}^-$, which are the graphs obtained from $K_{2n}$ and $K_{n,n}$ by deleting the edges of a perfect matching. The biclique covering number of $K_{2n}^-$ is $\lceil \log_2 n \rceil$ (see~\cite{moazami18}) and the biclique covering number of $K_{n,n}^-$ is the smallest $k$ for which $n \leq \binom{k}{\lfloor \nicefrac{k}{2} \rfloor}$ (see~\cite{FJKK07,BFRK08}).  

In this paper, we determine the biclique covering number for all grids. Let $G_{p, q}$ be the $p \times q$ grid.  Recall that $G_{p,q}$ has vertex set $[q] \times [p]$, where $(a,b)$ is adjacent to $(a',b')$ if and only if $|a-a'|+|b-b'|=1$.  The following is our main result. 

\begin{thm} \label{thm:main}
For all integers $1 \le p \le q$,
\[
\bc(G_{p,q}) = 
\begin{cases} 
\frac{pq}{2} -1, & \text{if $p$ is even and $q-1 = k (p-1) + 2 \ell$ for some integers $0 \le \ell < k$;} \\[6pt]
\big\lfloor \frac{pq}{2} \big\rfloor, & \text{otherwise.} 
\end{cases}
\]
\end{thm}

Since $G_{p,q} \simeq G_{q,p}$, our main result determines the biclique covering number of all grid graphs.  This settles an open problem raised by Denis Cornaz at the $9$\textsuperscript{th} Cargèse Workshop on Combinatorial Optimization in $2018$. Note that for $p$ even and $q \geq p(p-1)$, Theorem~\ref{thm:main} implies that $\bc(G_{p,q}) = \nicefrac{pq}{2}-1$.

\section{Main result} \label{sec:the_proof}
Since we are only interested in biclique covers of minimum size, we may assume that biclique covers only consist of maximal bicliques (under edge inclusion). For brevity, we call such biclique covers simply \emph{covers}.  
% If $C$ is a cover with a biclique $B$ whose edges are strictly contained in a biclique $B'$, we may replace $B$ and $B'$ to obtain a cover $C'$. Thus if there exists a cover with $c$ bicliques, there exists a cover with $c$ bibliques consisting of maximal bicliques.
% Moreover, 
% From this point, we will restrict to covers which  consist only of maximal bicliques under edge inclusion.  
For $G_{p,q}$, this implies that covers consist of elements which are isomorphic to  $K_{1,3}, K_{1,4}$, or the $4$-cycle.  
We first establish the following upper bound. 

\begin{lem} \label{lem:checkerboard}
For all integers $1 \leqslant p \leqslant q$, $\bc(G_{p,q}) \leq \lfloor \nicefrac{pq}{2} \rfloor$.  
\end{lem}

\begin{proof}
Observe that $G_{p,q}$ is a bipartite graph with bipartition $(X,Y)$ where $|X|=\lfloor \nicefrac{pq}{2} \rfloor$ and $|Y|=\lceil \nicefrac{pq}{2} \rceil$.  Therefore taking the set of stars centered at vertices in $X$ gives a cover of size $\lfloor \nicefrac{pq}{2} \rfloor$. See Figure~\ref{fig:checkerboard}. 
\end{proof}

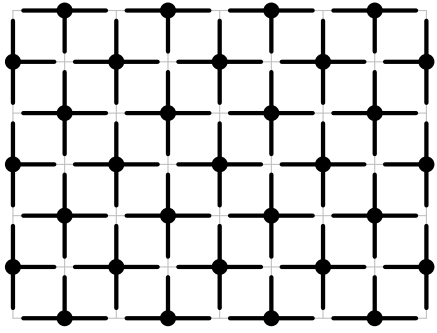
\begin{figure}[ht!]
\scalebox{.8}{\begin{tikzpicture}[scale=.85]
\tikzset{dot/.style = {draw,fill,circle, inner sep = 1.95pt},
frame/.style = {draw=lightgray,line width=.5pt},
biclique/.style = {draw,line width=2pt},
brace/.style = {decorate,decoration={brace,amplitude=10pt},xshift=0pt,yshift=0pt,line width=1.5pt}
}
    \def\e{.2}
    \def\p{6}
    \def\q{8}
    \def\w{.1}
    \def\h{.3}

    \foreach \x in {0,...,\q} {
        \draw [frame] (\x,0) -- (\x,\p);
    }

    \foreach \y in {0,...,\p} {
        \draw [frame] (0,\y) -- (\q,\y);
    }

    %%%%% K_{1,3}
    \foreach \x in {1,3,5,7} {
        \begin{scope}[biclique,shift={(\x,0)},rotate=0]
        \node[dot] at (0,0) {};
        \draw (-1+\e,0) -- (1-\e,0);
        \draw (0,0) -- (0,1-\e);
        \end{scope}
        
        \begin{scope}[biclique,shift={(\x,\p)},rotate=180]
        \node[dot] at (0,0) {};
        \draw (-1+\e,0) -- (1-\e,0);
        \draw (0,0) -- (0,1-\e);
        \end{scope}
    }
    
    \foreach \y in {1,3,5} {
        \begin{scope}[biclique,shift={(0,\y)},rotate=-90]
        \node[dot] at (0,0) {};
        \draw (-1+\e,0) -- (1-\e,0);
        \draw (0,0) -- (0,1-\e);
        \end{scope}
        
        \begin{scope}[biclique,shift={(\q,\y)},rotate=90]
        \node[dot] at (0,0) {};
        \draw (-1+\e,0) -- (1-\e,0);
        \draw (0,0) -- (0,1-\e);
        \end{scope}
    }
    
    %%%%% K_{1,4}
    \foreach \x in {2,4,6} {
        \foreach \y in {1,3,5} {
            \begin{scope}[biclique,shift={(\x,\y)}]
            \node[dot] at (0,0) {};
            \draw (-1+\e,0) -- (1-\e,0);
            \draw (0,-1+\e) -- (0,1-\e);
            \end{scope}
        }
    }
    
    \foreach \x in {1,3,5,7} {
        \foreach \y in {2,4} {
            \begin{scope}[biclique,shift={(\x,\y)}]
            \node[dot] at (0,0) {};
            \draw (-1+\e,0) -- (1-\e,0);
            \draw (0,-1+\e) -- (0,1-\e);
            \end{scope}
        }
    }

    %    \foreach \x in {0,...,3} {
%    	\foreach \y in {0,...,3} {
%    		\node[dot] at (\x,\y) {};
%    	}
%    }

%    \foreach \x in {0,...,3} {
%    	\foreach \y in {0,...,3} {
%    		\node[dot] at (\x,\y) {};
%    	}
%    }
    
\end{tikzpicture}}
\caption{The stars centered at black vertices are a cover.}
\label{fig:checkerboard}
\end{figure}

We now establish the following lower bounds.  

\begin{lem} \label{lem:lower}
For all integers $1 \le p \le q$, \\[3pt]
$
\begin{array}{r@{\hskip3pt}c@{\hskip3pt}ll}
    \big\lfloor \nicefrac{pq}{2} \big\rfloor &\le& \bc(G_{p,q}), & \text{if $p$ is odd.} \\[6pt]
    \nicefrac{pq}{2} - 1 &\le& \bc(G_{p,q}), & \text{if $p$ is even.}
\end{array}
$
% \[
% \bc(G_{p,q}) \ge 
% \begin{cases} 
% \frac{pq}{2} -1, & \text{if $p$ is even} \\[6pt]
% \big\lfloor \frac{pq}{2} \big\rfloor, & \text{if $p$ is odd}
% \end{cases}
% \]
\end{lem}

\begin{proof}
Let $1 \le p \le q$. We define a special subset $S(G_{p,q})$ of edges of $G_{p,q}$ inductively as follows.  If $p \in \{1,2\}$, we let $S(G_{p,q})$ be the set of all horizontal edges of $G_{p,q}$.  If $p \geq 3$, let $\out$ be the outer cycle of $G_{p,q}$ and define $S(G_{p,q})=E(\out) \cup S(G_{p,q}-V(\out))$. See Figure~\ref{fig:p_even} and Figure~\ref{fig:p_odd}. An easy induction gives 
\begin{enumerate}
    \item $|S(G_{p,q})|=pq-1$, if $p$ is odd.
    \item $|S(G_{p,q})|=pq-2$, if $p$ is even.
\end{enumerate}
\begin{figure}[ht!]\centering
\begin{subfigure}{.4\textwidth}\centering
    \scalebox{.8}{\begin{tikzpicture}[scale=.8]
\tikzset{frame/.style = {draw=lightgray,line width=.5pt}
}
    \def\e{.2}
    \def\p{5}
    \def\q{7}
    
    \clip (-2*\e,-2*\e-.5) rectangle (\q + 2*\e,\p + 2*\e+.5);

    \foreach \x in {0,...,\q} {
        \draw [frame] (\x,0) -- (\x,\p);
    }

    \foreach \y in {0,...,\p} {
        \draw [frame] (0,\y) -- (\q,\y);
    }
    
    \foreach \z in {0,1} {
        \draw[line width=2pt](\z,\z)--(\z,\p-\z)--(\q-\z,\p-\z)--(\q-\z,\z)--cycle;
    }
    
    \draw[line width=2pt](2,2)--(\q-2,2);
    \draw[line width=2pt](2,\p-2)--(\q-2,\p-2);

\end{tikzpicture}}
    \caption{$S(G_{p,q})$ for $p$ even.}
    \label{fig:p_even}
\end{subfigure}\hspace{1cm}
\begin{subfigure}{.4\textwidth}\centering
    \scalebox{.8}{\begin{tikzpicture}[scale=.8]
\tikzset{frame/.style = {draw=lightgray,line width=.5pt}
}
    \def\e{.2}
    \def\p{6}
    \def\q{7}
    
    \clip (-2*\e,-2*\e) rectangle (\q + 2*\e,\p + 2*\e);

    \foreach \x in {0,...,\q} {
        \draw [frame] (\x,0) -- (\x,\p);
    }

    \foreach \y in {0,...,\p} {
        \draw [frame] (0,\y) -- (\q,\y);
    }
    
    \foreach \z in {0,...,3} {
        \draw[line width=2pt](\z,\z)--(\z,\p-\z)--(\q-\z,\p-\z)--(\q-\z,\z)--cycle;
    }

\end{tikzpicture}}
    \caption{$S(G_{p,q})$ for $p$ odd.}
    \label{fig:p_odd}
\end{subfigure}
\caption{}
\label{fig:S_G_pq}
\end{figure}

On the other hand, every biclique of $G_{p,q}$ contains at most $2$ edges of $S(G_{p,q})$.  Therefore, $\bc(G_{p,q}) \geq  \lceil \nicefrac{|S(G_{p,q})|}{2} \rceil$, which completes the proof. 
\end{proof}

Next, we give values of $p$ and $q$ for which $\bc(G_{p,q}) = \lfloor \nicefrac{pq}{2} \rfloor-1$.

\begin{lem} \label{lem:constructions}
$\bc(G_{p,q}) = \nicefrac{pq}{2} -1$ if $p$ is even and $q-1 = k (p-1) + 2 \ell$ for some integers $0 \le \ell < k $.
\end{lem}

\begin{proof}
By Lemma~\ref{lem:lower}, it suffices to construct a cover $\mathcal{C}$ of size $\nicefrac{pq}{2} -1$.
Since $q-1 = k (p-1) + 2 \ell$ for some integers $0 \le \ell < k $, we can decompose $G_{p,q}$ as $k$ copies of $G_{p,p}$ and $\ell$ copies of $G_{p,3}$ such that each of the $\ell$ copies of $G_{p,3}$ is between two copies of $G_{p,p}$.  Since $p$ is even, $G_{p,p}$ has two covers of size $\nicefrac{p^2}{2}-1$, whose set of $4$-cycles are the $4$-cycles of the two diagonals of $G_{p,p}$, respectively.  See Figure~\ref{fig:grid-6_6_blocks} for the case $p = 6$. 
 
 The cover $\mathcal{C}$ is constructed as follows. For every two consecutive copies of $G_{p,p}$, we use one cover from Figure~\ref{fig:grid-6_6_blocks} on one copy and the other cover from Figure~\ref{fig:grid-6_6_blocks} on the other copy.  Note that this results in $\nicefrac{p}{2}-1$ bicliques that are in both copies. For every copy of $G_{p,3}$ between two copies of $G_{p,p}$ we again use the two covers of size $\nicefrac{p^2}{2}-1$ on the two copies of $G_{p,p}$, and then we use $\nicefrac{p}{2}+1$ bicliques in $G_{p,3}$ to cover the remaining edges, see Figure~\ref{fig:grid_example_1}.  The total construction when $p = 6$ and $q = 25$ is given in Figure~\ref{fig:grid_example_2}. Observe that:
\[
|\mathcal{C}| = k\Big( \frac{p^2}{2}-1 \Big) + \ell\Big( \frac{p}{2}+1 \Big) - (k - \ell -1)\Big( \frac{p}{2}-1 \Big) = \frac{p}{2}\big( k(p-1)+2\ell+1 \big)-1= \frac{pq}{2} - 1\,. \qedhere
\]
\end{proof}

\begin{figure}[ht!]
\begin{subfigure}{.4\textwidth}\centering
	\scalebox{.8}{\begin{tikzpicture}[scale=1]
\tikzset{dot/.style = {draw,fill,circle, inner sep = 2pt},
frame/.style = {draw=lightgray,line width=.5pt},
biclique/.style = {draw,line width=2pt},
brace/.style = {decorate,decoration={brace,amplitude=10pt},xshift=0pt,yshift=0pt,line width=1.5pt}
}
    \def\e{.2}
    \def\p{5}
    \def\q{5}
    \def\w{.1}
    \def\h{.3}

    \foreach \x in {0,...,\q} {
        \draw [frame] (\x,0) -- (\x,\p);
    }

    \foreach \y in {0,...,\p} {
        \draw [frame] (0,\y) -- (\q,\y);
    }

    \begin{scope}[biclique,shift={(0,4)}]
    \draw[fill=gray!30] (0,0) -- (1,0) -- (1,1) -- (0,1) -- cycle;
    \end{scope}

    \begin{scope}[biclique,shift={(1,3)}]
    \draw[fill=gray!30] (0,0) -- (1,0) -- (1,1) -- (0,1) -- cycle;
    \end{scope}
    
    \begin{scope}[biclique,shift={(2,2)}]
    \draw[fill=gray!30] (0,0) -- (1,0) -- (1,1) -- (0,1) -- cycle;
    \end{scope}

    \begin{scope}[biclique,shift={(3,1)}]
    \draw[fill=gray!30] (0,0) -- (1,0) -- (1,1) -- (0,1) -- cycle;
    \end{scope}
    
     \begin{scope}[biclique,shift={(4,0)}]
    \draw[fill=gray!30] (0,0) -- (1,0) -- (1,1) -- (0,1) -- cycle;
    \end{scope}

    \begin{scope}[biclique,shift={(1,0)},rotate=0]
    \node[dot] at (0,0) {};
    \draw (-1+\e,0) -- (1-\e,0);
    \draw (0,0) -- (0,1-\e);
    \end{scope}
    
    \begin{scope}[biclique,shift={(3,0)},rotate=0]
    \node[dot] at (0,0) {};
    \draw (-1+\e,0) -- (1-\e,0);
    \draw (0,0) -- (0,1-\e);
    \end{scope}

    \begin{scope}[biclique,shift={(0,1)},rotate=0]
    \node[dot] at (0,0) {};
    \draw (0,0) -- (1-\e,0);
    \draw (0,-1+\e) -- (0,1-\e);
    \end{scope}
    
    \begin{scope}[biclique,shift={(0,3)},rotate=0]
    \node[dot] at (0,0) {};
    \draw (0,0) -- (1-\e,0);
    \draw (0,-1+\e) -- (0,1-\e);
    \end{scope}

    \begin{scope}[biclique,shift={(2,\p)},rotate=180]
    \node[dot] at (0,0) {};
    \draw (-1+\e,0) -- (1-\e,0);
    \draw (0,0) -- (0,1-\e);
    \end{scope}
    
    \begin{scope}[biclique,shift={(4,\p)},rotate=180]
    \node[dot] at (0,0) {};
    \draw (-1+\e,0) -- (1-\e,0);
    \draw (0,0) -- (0,1-\e);
    \end{scope}
    
    \begin{scope}[biclique,shift={(1,2)}]
    \node[dot] at (0,0) {};
    \draw (-1+\e,0) -- (1-\e,0);
    \draw (0,-1+\e) -- (0,1-\e);
    \end{scope}
    \begin{scope}[biclique,shift={(2,1)}]
    \node[dot] at (0,0) {};
    \draw (-1+\e,0) -- (1-\e,0);
    \draw (0,-1+\e) -- (0,1-\e);
    \end{scope}
    \begin{scope}[biclique,shift={(4,3)}]
    \node[dot] at (0,0) {};
    \draw (-1+\e,0) -- (1-\e,0);
    \draw (0,-1+\e) -- (0,1-\e);
    \end{scope}
    \begin{scope}[biclique,shift={(3,4)}]
    \node[dot] at (0,0) {};
    \draw (-1+\e,0) -- (1-\e,0);
    \draw (0,-1+\e) -- (0,1-\e);
    \end{scope}

    \begin{scope}[biclique,shift={(\q,2)}]
    \node[dot] at (0,0) {};
    \draw (-1+\e,0) -- (0,0);
    \draw (0,-1+\e) -- (0,1-\e);
    \end{scope}
    
    \begin{scope}[biclique,shift={(\q,4)}]
    \node[dot] at (0,0) {};
    \draw (-1+\e,0) -- (0,0);
    \draw (0,-1+\e) -- (0,1-\e);
    \end{scope}
    
\end{tikzpicture}}
	\caption{}
	\label{fig:grid-6_6_block_1}
\end{subfigure}\hspace{1cm}
\begin{subfigure}{.4\textwidth}\centering
	\scalebox{.8}{\begin{tikzpicture}[scale=1]
\tikzset{dot/.style = {draw,fill,circle, inner sep = 2pt},
frame/.style = {draw=lightgray,line width=.5pt},
biclique/.style = {draw,line width=2pt},
brace/.style = {decorate,decoration={brace,amplitude=10pt},xshift=0pt,yshift=0pt,line width=1.5pt}
}
    \def\e{.2}
    \def\p{5}
    \def\q{5}
    \def\w{.1}
    \def\h{.3}

    \foreach \x in {0,...,\q} {
        \draw [frame] (\x,0) -- (\x,\p);
    }

    \foreach \y in {0,...,\p} {
        \draw [frame] (0,\y) -- (\q,\y);
    }

\begin{scope}[shift={(\q,0)},xscale=-1]
    \begin{scope}[biclique,shift={(0,4)}]
    \draw[fill=gray!30] (0,0) -- (1,0) -- (1,1) -- (0,1) -- cycle;
    \end{scope}

    \begin{scope}[biclique,shift={(1,3)}]
    \draw[fill=gray!30] (0,0) -- (1,0) -- (1,1) -- (0,1) -- cycle;
    \end{scope}
    
    \begin{scope}[biclique,shift={(2,2)}]
    \draw[fill=gray!30] (0,0) -- (1,0) -- (1,1) -- (0,1) -- cycle;
    \end{scope}

    \begin{scope}[biclique,shift={(3,1)}]
    \draw[fill=gray!30] (0,0) -- (1,0) -- (1,1) -- (0,1) -- cycle;
    \end{scope}
    
     \begin{scope}[biclique,shift={(4,0)}]
    \draw[fill=gray!30] (0,0) -- (1,0) -- (1,1) -- (0,1) -- cycle;
    \end{scope}

    \begin{scope}[biclique,shift={(1,0)},rotate=0]
    \node[dot] at (0,0) {};
    \draw (-1+\e,0) -- (1-\e,0);
    \draw (0,0) -- (0,1-\e);
    \end{scope}
    
    \begin{scope}[biclique,shift={(3,0)},rotate=0]
    \node[dot] at (0,0) {};
    \draw (-1+\e,0) -- (1-\e,0);
    \draw (0,0) -- (0,1-\e);
    \end{scope}

    \begin{scope}[biclique,shift={(0,1)},rotate=0]
    \node[dot] at (0,0) {};
    \draw (0,0) -- (1-\e,0);
    \draw (0,-1+\e) -- (0,1-\e);
    \end{scope}
    
    \begin{scope}[biclique,shift={(0,3)},rotate=0]
    \node[dot] at (0,0) {};
    \draw (0,0) -- (1-\e,0);
    \draw (0,-1+\e) -- (0,1-\e);
    \end{scope}

    \begin{scope}[biclique,shift={(2,\p)},rotate=180]
    \node[dot] at (0,0) {};
    \draw (-1+\e,0) -- (1-\e,0);
    \draw (0,0) -- (0,1-\e);
    \end{scope}
    
    \begin{scope}[biclique,shift={(4,\p)},rotate=180]
    \node[dot] at (0,0) {};
    \draw (-1+\e,0) -- (1-\e,0);
    \draw (0,0) -- (0,1-\e);
    \end{scope}
    
    \begin{scope}[biclique,shift={(1,2)}]
    \node[dot] at (0,0) {};
    \draw (-1+\e,0) -- (1-\e,0);
    \draw (0,-1+\e) -- (0,1-\e);
    \end{scope}
    \begin{scope}[biclique,shift={(2,1)}]
    \node[dot] at (0,0) {};
    \draw (-1+\e,0) -- (1-\e,0);
    \draw (0,-1+\e) -- (0,1-\e);
    \end{scope}
    \begin{scope}[biclique,shift={(4,3)}]
    \node[dot] at (0,0) {};
    \draw (-1+\e,0) -- (1-\e,0);
    \draw (0,-1+\e) -- (0,1-\e);
    \end{scope}
    \begin{scope}[biclique,shift={(3,4)}]
    \node[dot] at (0,0) {};
    \draw (-1+\e,0) -- (1-\e,0);
    \draw (0,-1+\e) -- (0,1-\e);
    \end{scope}

    \begin{scope}[biclique,shift={(\q,2)}]
    \node[dot] at (0,0) {};
    \draw (-1+\e,0) -- (0,0);
    \draw (0,-1+\e) -- (0,1-\e);
    \end{scope}
    
    \begin{scope}[biclique,shift={(\q,4)}]
    \node[dot] at (0,0) {};
    \draw (-1+\e,0) -- (0,0);
    \draw (0,-1+\e) -- (0,1-\e);
    \end{scope}
    
\end{scope}
    
\end{tikzpicture}}
	\caption{}
	\label{fig:grid-6_6_block_2}
\end{subfigure}
\caption{Two covers of $G_{6,6}$ of size 17.}
\label{fig:grid-6_6_blocks}
\end{figure}
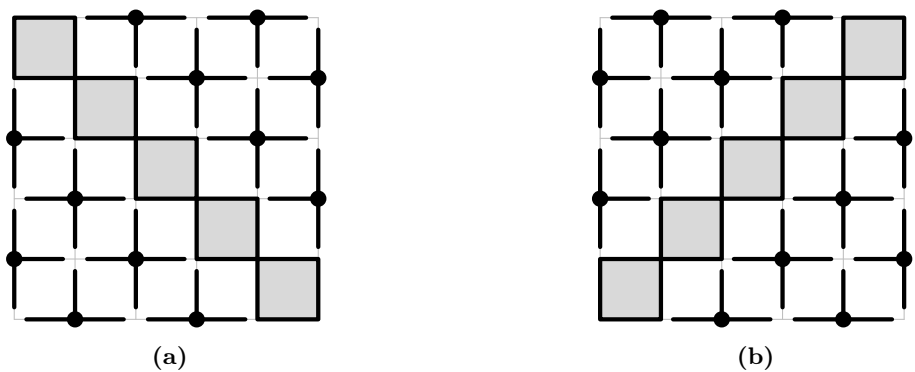

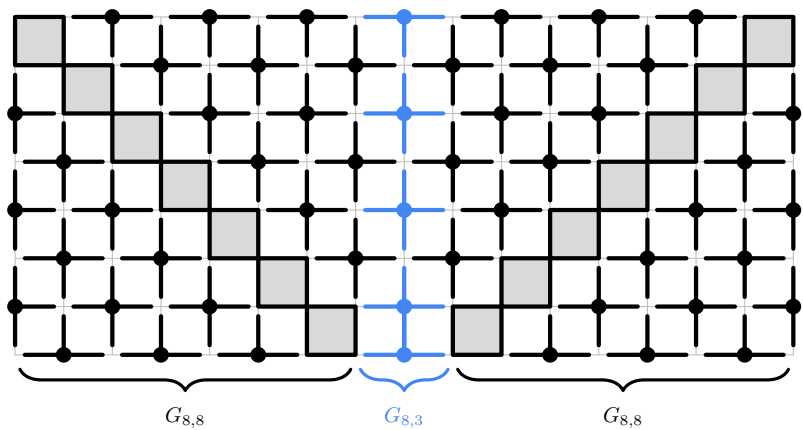
\begin{figure}[ht!]
\scalebox{.8}{\begin{tikzpicture}[scale=.8]
\tikzset{dot/.style = {draw,fill,circle, inner sep = 1.9pt},
frame/.style = {draw=lightgray,line width=.5pt},
biclique/.style = {draw,line width=2pt},
brace/.style = {decorate,decoration={brace,amplitude=10pt},xshift=0pt,yshift=0pt,line width=1.5pt}
}
    \def\e{.2}
    \def\p{7}
    \def\q{16}
    \def\w{.1}
    \def\h{.3}

    \foreach \x in {0,...,\q} {
        \draw [frame] (\x,0) -- (\x,\p);
    }

    \foreach \y in {0,...,\p} {
        \draw [frame] (0,\y) -- (\q,\y);
    }
    
\begin{scope}[]
    \begin{scope}[biclique,shift={(0,6)}]
    \draw[fill=gray!30] (0,0) -- (1,0) -- (1,1) -- (0,1) -- cycle;
    \end{scope}
    
    \begin{scope}[biclique,shift={(1,5)}]
    \draw[fill=gray!30] (0,0) -- (1,0) -- (1,1) -- (0,1) -- cycle;
    \end{scope}
    
    \begin{scope}[biclique,shift={(2,4)}]
    \draw[fill=gray!30] (0,0) -- (1,0) -- (1,1) -- (0,1) -- cycle;
    \end{scope}
    
    \begin{scope}[biclique,shift={(3,3)}]
    \draw[fill=gray!30] (0,0) -- (1,0) -- (1,1) -- (0,1) -- cycle;
    \end{scope}
    
    \begin{scope}[biclique,shift={(4,2)}]
    \draw[fill=gray!30] (0,0) -- (1,0) -- (1,1) -- (0,1) -- cycle;
    \end{scope}
    
    \begin{scope}[biclique,shift={(5,1)}]
    \draw[fill=gray!30] (0,0) -- (1,0) -- (1,1) -- (0,1) -- cycle;
    \end{scope}
    
    \begin{scope}[biclique,shift={(6,0)}]
    \draw[fill=gray!30] (0,0) -- (1,0) -- (1,1) -- (0,1) -- cycle;
    \end{scope}
    
    %%%%% K_{1,3}
    \foreach \x in {2,4,6} {
        \begin{scope}[biclique,shift={(\x,7)},rotate=180]
        \node[dot] at (0,0) {};
        \draw (-1+\e,0) -- (1-\e,0);
        \draw (0,0) -- (0,1-\e);
        \end{scope}
    }
    
    \foreach \x in {1,3,5} {
        \begin{scope}[biclique,shift={(\x,0)},rotate=0]
        \node[dot] at (0,0) {};
        \draw (-1+\e,0) -- (1-\e,0);
        \draw (0,0) -- (0,1-\e);
        \end{scope}
    }
    
    \foreach \y in {1,3,5} {
        \begin{scope}[biclique,shift={(0,\y)},rotate=-90]
        \node[dot] at (0,0) {};
        \draw (-1+\e,0) -- (1-\e,0);
        \draw (0,0) -- (0,1-\e);
        \end{scope}
    }
    
    %%%%% K_{1,4}
    \begin{scope}[biclique,shift={(1,2)}]
    \node[dot] at (0,0) {};
    \draw (-1+\e,0) -- (1-\e,0);
    \draw (0,-1+\e) -- (0,1-\e);
    \end{scope}
    
    \begin{scope}[biclique,shift={(2,1)}]
    \node[dot] at (0,0) {};
    \draw (-1+\e,0) -- (1-\e,0);
    \draw (0,-1+\e) -- (0,1-\e);
    \end{scope}
    
    \begin{scope}[biclique,shift={(1,4)}]
    \node[dot] at (0,0) {};
    \draw (-1+\e,0) -- (1-\e,0);
    \draw (0,-1+\e) -- (0,1-\e);
    \end{scope}
    
    \begin{scope}[biclique,shift={(2,3)}]
    \node[dot] at (0,0) {};
    \draw (-1+\e,0) -- (1-\e,0);
    \draw (0,-1+\e) -- (0,1-\e);
    \end{scope}
    
    \begin{scope}[biclique,shift={(3,2)}]
    \node[dot] at (0,0) {};
    \draw (-1+\e,0) -- (1-\e,0);
    \draw (0,-1+\e) -- (0,1-\e);
    \end{scope}
    
    \begin{scope}[biclique,shift={(4,1)}]
    \node[dot] at (0,0) {};
    \draw (-1+\e,0) -- (1-\e,0);
    \draw (0,-1+\e) -- (0,1-\e);
    \end{scope}
    
    \begin{scope}[biclique,shift={(3,6)}]
    \node[dot] at (0,0) {};
    \draw (-1+\e,0) -- (1-\e,0);
    \draw (0,-1+\e) -- (0,1-\e);
    \end{scope}
    
    \begin{scope}[biclique,shift={(4,5)}]
    \node[dot] at (0,0) {};
    \draw (-1+\e,0) -- (1-\e,0);
    \draw (0,-1+\e) -- (0,1-\e);
    \end{scope}
    
    \begin{scope}[biclique,shift={(5,4)}]
    \node[dot] at (0,0) {};
    \draw (-1+\e,0) -- (1-\e,0);
    \draw (0,-1+\e) -- (0,1-\e);
    \end{scope}
    
    \begin{scope}[biclique,shift={(6,3)}]
    \node[dot] at (0,0) {};
    \draw (-1+\e,0) -- (1-\e,0);
    \draw (0,-1+\e) -- (0,1-\e);
    \end{scope}
    
    \begin{scope}[biclique,shift={(7,2)}]
    \node[dot] at (0,0) {};
    \draw (-1+\e,0) -- (1-\e,0);
    \draw (0,-1+\e) -- (0,1-\e);
    \end{scope}
    
    \begin{scope}[biclique,shift={(7,4)}]
    \node[dot] at (0,0) {};
    \draw (-1+\e,0) -- (1-\e,0);
    \draw (0,-1+\e) -- (0,1-\e);
    \end{scope}
    
    \begin{scope}[biclique,shift={(7,6)}]
    \node[dot] at (0,0) {};
    \draw (-1+\e,0) -- (1-\e,0);
    \draw (0,-1+\e) -- (0,1-\e);
    \end{scope}
    
    \begin{scope}[biclique,shift={(6,5)}]
    \node[dot] at (0,0) {};
    \draw (-1+\e,0) -- (1-\e,0);
    \draw (0,-1+\e) -- (0,1-\e);
    \end{scope}
    
    \begin{scope}[biclique,shift={(5,6)}]
    \node[dot] at (0,0) {};
    \draw (-1+\e,0) -- (1-\e,0);
    \draw (0,-1+\e) -- (0,1-\e);
    \end{scope}
    
    \draw [brace] (\p-\w,-\h) -- (0+\w,-\h) node [midway,yshift=-8mm] {$G_{8,8}$};
\end{scope}

\begin{scope}[shift={(8,0)},lightblue]
    \begin{scope}[biclique,shift={(0,1)}]
    \node[dot] at (0,0) {};
    \draw (-1+\e,0) -- (1-\e,0);
    \draw (0,-1+\e) -- (0,1-\e);
    \end{scope}
    
    \begin{scope}[biclique,shift={(0,3)}]
    \node[dot] at (0,0) {};
    \draw (-1+\e,0) -- (1-\e,0);
    \draw (0,-1+\e) -- (0,1-\e);
    \end{scope}
    
    \begin{scope}[biclique,shift={(0,5)}]
    \node[dot] at (0,0) {};
    \draw (-1+\e,0) -- (1-\e,0);
    \draw (0,-1+\e) -- (0,1-\e);
    \end{scope}
    
    \begin{scope}[biclique,shift={(0,7)},rotate=180]
    \node[dot] at (0,0) {};
    \draw (-1+\e,0) -- (1-\e,0);
    \draw (0,0) -- (0,1-\e);
    \end{scope}
    
    \begin{scope}[biclique]
    \node[dot] at (0,0) {};
    \draw (-1+\e,0) -- (1-\e,0);
    \draw (0,0) -- (0,1-\e);
    \end{scope}
    
    \draw [brace] (1-\w,-\h) -- (-1+\w,-\h) node [lightblue,midway,yshift=-8mm] {$G_{8,3}$};
\end{scope}

\begin{scope}[shift={(\q,0)},xscale=-1]
    \begin{scope}[biclique,shift={(0,6)}]
    \draw[fill=gray!30] (0,0) -- (1,0) -- (1,1) -- (0,1) -- cycle;
    \end{scope}
    
    \begin{scope}[biclique,shift={(1,5)}]
    \draw[fill=gray!30] (0,0) -- (1,0) -- (1,1) -- (0,1) -- cycle;
    \end{scope}
    
    \begin{scope}[biclique,shift={(2,4)}]
    \draw[fill=gray!30] (0,0) -- (1,0) -- (1,1) -- (0,1) -- cycle;
    \end{scope}
    
    \begin{scope}[biclique,shift={(3,3)}]
    \draw[fill=gray!30] (0,0) -- (1,0) -- (1,1) -- (0,1) -- cycle;
    \end{scope}
    
    \begin{scope}[biclique,shift={(4,2)}]
    \draw[fill=gray!30] (0,0) -- (1,0) -- (1,1) -- (0,1) -- cycle;
    \end{scope}
    
    \begin{scope}[biclique,shift={(5,1)}]
    \draw[fill=gray!30] (0,0) -- (1,0) -- (1,1) -- (0,1) -- cycle;
    \end{scope}
    
    \begin{scope}[biclique,shift={(6,0)}]
    \draw[fill=gray!30] (0,0) -- (1,0) -- (1,1) -- (0,1) -- cycle;
    \end{scope}
    
    %%%%% K_{1,3}
    \foreach \x in {2,4,6} {
        \begin{scope}[biclique,shift={(\x,7)},rotate=180]
        \node[dot] at (0,0) {};
        \draw (-1+\e,0) -- (1-\e,0);
        \draw (0,0) -- (0,1-\e);
        \end{scope}
    }
    
    \foreach \x in {1,3,5} {
        \begin{scope}[biclique,shift={(\x,0)},rotate=0]
        \node[dot] at (0,0) {};
        \draw (-1+\e,0) -- (1-\e,0);
        \draw (0,0) -- (0,1-\e);
        \end{scope}
    }
    
    \foreach \y in {1,3,5} {
        \begin{scope}[biclique,shift={(0,\y)},rotate=-90]
        \node[dot] at (0,0) {};
        \draw (-1+\e,0) -- (1-\e,0);
        \draw (0,0) -- (0,1-\e);
        \end{scope}
    }
    
    %%%%% K_{1,4}
    \begin{scope}[biclique,shift={(1,2)}]
    \node[dot] at (0,0) {};
    \draw (-1+\e,0) -- (1-\e,0);
    \draw (0,-1+\e) -- (0,1-\e);
    \end{scope}
    
    \begin{scope}[biclique,shift={(2,1)}]
    \node[dot] at (0,0) {};
    \draw (-1+\e,0) -- (1-\e,0);
    \draw (0,-1+\e) -- (0,1-\e);
    \end{scope}
    
    \begin{scope}[biclique,shift={(1,4)}]
    \node[dot] at (0,0) {};
    \draw (-1+\e,0) -- (1-\e,0);
    \draw (0,-1+\e) -- (0,1-\e);
    \end{scope}
    
    \begin{scope}[biclique,shift={(2,3)}]
    \node[dot] at (0,0) {};
    \draw (-1+\e,0) -- (1-\e,0);
    \draw (0,-1+\e) -- (0,1-\e);
    \end{scope}
    
    \begin{scope}[biclique,shift={(3,2)}]
    \node[dot] at (0,0) {};
    \draw (-1+\e,0) -- (1-\e,0);
    \draw (0,-1+\e) -- (0,1-\e);
    \end{scope}
    
    \begin{scope}[biclique,shift={(4,1)}]
    \node[dot] at (0,0) {};
    \draw (-1+\e,0) -- (1-\e,0);
    \draw (0,-1+\e) -- (0,1-\e);
    \end{scope}
    
    \begin{scope}[biclique,shift={(3,6)}]
    \node[dot] at (0,0) {};
    \draw (-1+\e,0) -- (1-\e,0);
    \draw (0,-1+\e) -- (0,1-\e);
    \end{scope}
    
    \begin{scope}[biclique,shift={(4,5)}]
    \node[dot] at (0,0) {};
    \draw (-1+\e,0) -- (1-\e,0);
    \draw (0,-1+\e) -- (0,1-\e);
    \end{scope}
    
    \begin{scope}[biclique,shift={(5,4)}]
    \node[dot] at (0,0) {};
    \draw (-1+\e,0) -- (1-\e,0);
    \draw (0,-1+\e) -- (0,1-\e);
    \end{scope}
    
    \begin{scope}[biclique,shift={(6,3)}]
    \node[dot] at (0,0) {};
    \draw (-1+\e,0) -- (1-\e,0);
    \draw (0,-1+\e) -- (0,1-\e);
    \end{scope}
    
    \begin{scope}[biclique,shift={(7,2)}]
    \node[dot] at (0,0) {};
    \draw (-1+\e,0) -- (1-\e,0);
    \draw (0,-1+\e) -- (0,1-\e);
    \end{scope}
    
    \begin{scope}[biclique,shift={(7,4)}]
    \node[dot] at (0,0) {};
    \draw (-1+\e,0) -- (1-\e,0);
    \draw (0,-1+\e) -- (0,1-\e);
    \end{scope}
    
    \begin{scope}[biclique,shift={(7,6)}]
    \node[dot] at (0,0) {};
    \draw (-1+\e,0) -- (1-\e,0);
    \draw (0,-1+\e) -- (0,1-\e);
    \end{scope}
    
    \begin{scope}[biclique,shift={(6,5)}]
    \node[dot] at (0,0) {};
    \draw (-1+\e,0) -- (1-\e,0);
    \draw (0,-1+\e) -- (0,1-\e);
    \end{scope}
    
    \begin{scope}[biclique,shift={(5,6)}]
    \node[dot] at (0,0) {};
    \draw (-1+\e,0) -- (1-\e,0);
    \draw (0,-1+\e) -- (0,1-\e);
    \end{scope}
    
    \draw [brace] (0+\w,-\h)-- (\p-\w,-\h) node [midway,yshift=-8mm] {$G_{8,8}$};
\end{scope}

\end{tikzpicture}}
\caption{Cover of $G_{8,17}$ of size $67$.}
\label{fig:grid_example_1}
\end{figure}

\begin{figure}[ht!]
\scalebox{.635}{\input{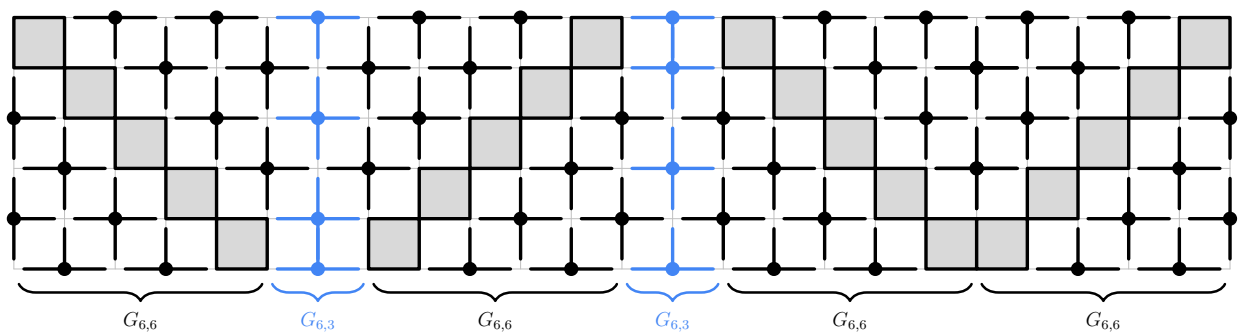}}
\caption{Cover of $G_{6,25}$ of size $74$.}
\label{fig:grid_example_2}
\end{figure}

We now establish that we may choose a cover with special properties, before completing the proof of the main theorem.  Recall that $\out$ is the outer cycle of $G_{p,q}$. Let $\mathcal C$ be a cover of $G_{p,q}$.  A \emph{boundary element} of $\mathcal{C}$ is an element of $\mathcal{C}$ containing at least one edge of $\out$.  A \emph{boundary $4$-cycle} is a boundary element that is a $4$-cycle and a \emph{boundary star} is a boundary element that is a star. 

\begin{lem}\label{lem:cover_props}
For every cover of $G_{p,q}$, there exists a cover $\mathcal{C}$ of the same size with the following properties:
\begin{enumerate}[label={\rm(\roman*)}]
    \item all boundary stars of $\mathcal{C}$ are pairwise edge-disjoint; \label{tag:item_1}
    \item no edge is contained in both a boundary $4$-cycle and a boundary star of $\mathcal{C}$.  \label{tag:item_2}
\end{enumerate}
\end{lem}

\begin{proof}
If two boundary stars meet in an edge, we replace them with a star and a $4$-cycle which cover a superset of edges, as in Figure~\ref{fig:2_stars}. 
If a boundary $4$-cycle and a boundary star are not edge-disjoint, then we replace them with two $4$-cycles which cover a superset of edges, as in Figure~\ref{fig:cube_star}.
\begin{figure}[ht!]
\begin{subfigure}{.4\textwidth}\centering
	\scalebox{.8}{\begin{tikzpicture}[scale=.85]
\tikzset{dot/.style = {draw,fill,circle, inner sep = 2pt},
frame/.style = {draw=lightgray,line width=.5pt},
biclique/.style = {draw,line width=2pt}}

    \def\e{.2}
    \def\w{.3}
    \def\p{2}
    \def\q{5}

    \foreach \x in {0,...,\q} {
        \draw [frame] (\x,-\w) -- (\x,\p);
    }

    \foreach \y in {0,...,\p} {
        \draw [frame] (-\w,\y) -- (\q+\w,\y);
    }
    
    \begin{scope}[biclique,shift={(2,2)},rotate=180]
    \node[dot] at (0,0) {}; 
    \draw (-1+\e,0) -- (1-\e,0);
    \draw (0,0) -- (0,1-\e);
    \end{scope}
    
    \begin{scope}[biclique,shift={(3,2)},rotate=180]
    \node[dot] at (0,0) {}; 
    \draw (-1+\e,0) -- (1-\e,0);
    \draw (0,0) -- (0,1-\e);
    \end{scope}

\end{tikzpicture}}
	\caption{}
	\label{fig:2_stars_1}
\end{subfigure}\hspace{1cm}
\begin{subfigure}{.4\textwidth}\centering
	\scalebox{.8}{\begin{tikzpicture}[scale=.85]
\tikzset{dot/.style = {draw,fill,circle, inner sep = 2pt},
frame/.style = {draw=lightgray,line width=.5pt},
biclique/.style = {draw,line width=2pt}}

    \def\e{.2}
    \def\w{.3}
    \def\p{2}
    \def\q{5}

    \foreach \x in {0,...,\q} {
        \draw [frame] (\x,-\w) -- (\x,\p);
    }

    \foreach \y in {0,...,\p} {
        \draw [frame] (-\w,\y) -- (\q+\w,\y);
    }
    
    % \foreach \x in {1,2,3} {
    %     \begin{scope}[biclique,shift={(\x,\p-1)}]
    %     \draw[fill=gray!30] (0,0) -- (1,0) -- (1,1) -- (0,1) -- cycle;
    %     \end{scope}
    % }

    \begin{scope}[biclique,shift={(2,\p)},rotate=180]
    \node[dot] at (0,0) {}; 
    \draw (-1+\e,0) -- (1-\e,0);
    \draw (0,0) -- (0,1-\e);
    \end{scope}
    
    \begin{scope}[biclique,shift={(3,\p-1)}]
    \draw[fill=gray!30] (0,0) -- (1,0) -- (1,1) -- (0,1) -- cycle;
    \end{scope}

\end{tikzpicture}}
	\caption{}
	\label{fig:2_stars_2}
\end{subfigure}
\caption{}
\label{fig:2_stars}
\end{figure}

\begin{figure}[ht!]
\begin{subfigure}{.4\textwidth}\centering
	\scalebox{.8}{\begin{tikzpicture}[scale=.85]
\tikzset{dot/.style = {draw,fill,circle, inner sep = 2pt},
frame/.style = {draw=lightgray,line width=.5pt},
biclique/.style = {draw,line width=2pt}}

    \def\e{.2}
    \def\w{.3}
    \def\p{2}
    \def\q{6}

    \foreach \x in {0,...,\q} {
        \draw [frame] (\x,-\w) -- (\x,\p);
    }

    \foreach \y in {0,...,\p} {
        \draw [frame] (-\w,\y) -- (\q+\w,\y);
    }
    
    \begin{scope}[biclique,shift={(2,\p-1)}]
    \draw[fill=gray!30] (0,0) -- (1,0) -- (1,1) -- (0,1) -- cycle;
    \end{scope}

    \begin{scope}[biclique,shift={(3,2)}]
    \node [dot] at (0,0) {};
    \draw (0,-1+\e) -- (0,0);
    \draw (-1+\e,0) -- (1-\e,0);
    \end{scope}

\end{tikzpicture}}
	\caption{}
	\label{fig:cube_star_1}
\end{subfigure}\hspace{1cm}
\begin{subfigure}{.4\textwidth}\centering
	\scalebox{.8}{\begin{tikzpicture}[scale=.85]
\tikzset{dot/.style = {draw,fill,circle, inner sep = 2pt},
frame/.style = {draw=lightgray,line width=.5pt},
biclique/.style = {draw,line width=2pt}}

    \def\e{.2}
    \def\w{.3}
    \def\p{2}
    \def\q{6}

    \foreach \x in {0,...,\q} {
        \draw [frame] (\x,-\w) -- (\x,\p);
    }

    \foreach \y in {0,...,\p} {
        \draw [frame] (-\w,\y) -- (\q+\w,\y);
    }
    
    \foreach \x in {2,3} {
        \begin{scope}[biclique,shift={(\x,\p-1)}]
        \draw[fill=gray!30] (0,0) -- (1,0) -- (1,1) -- (0,1) -- cycle;
        \end{scope}
    }

\end{tikzpicture}}
	\caption{}
	\label{fig:cube_star_2}
\end{subfigure}
\caption{}
\label{fig:cube_star}
\end{figure}
By repeatedly performing these two replacement rules, we eventually obtain a cover $\mathcal{C}$ satisfying~\ref{tag:item_1} and~\ref{tag:item_2}.
\end{proof}

We now complete the proof of Theorem~\ref{thm:main} by establishing the following converse to Lemma~\ref{lem:constructions}.
\begin{lem}\label{lem:lower_bound}
Let $1 \le p \le q$, with $p$ even. If $G_{p,q}$ has a cover of size $\nicefrac{pq}{2}-1$, then $q-1 = k (p-1) + 2 \ell$ for some integers $0 \le \ell < k$.
\end{lem}

\begin{proof}
The lemma clearly holds if $p=2$, so we may assume $p \geq 4$.  Let $\mathcal{C}$ be a cover of $G_{p,q}$ of size $\nicefrac{pq}{2}-1$ satisfying properties~\ref{tag:item_1} and~\ref{tag:item_2} of Lemma \ref{lem:cover_props}.  

We begin by defining some objects required for the proof.  Let $\out$ be the outer cycle of $G_{p,q}$.  The \emph{corners} of $G_{p,q}$ are the vertices $(1,1)$, $(1,p)$, $(q,1)$, and $(q,p)$.   Let $H$ be the subgraph of $G_{p,q}$ induced by the edges of the boundary $4$-cycles of $\mathcal C$.  A \emph{fence} is a connected component of $H$.  The \emph{size} of a fence is the number of boundary $4$-cycles it contains.  A \emph{link} is a connected component of $\out \setminus E(H)$ containing at least two vertices and no corners.  See Figure~\ref{fig:links} for an illustration of fences and links. 

Suppose $\mathsf{L}$ is a link contained on the topmost path of $G_{p,q}$.  Let $\mathsf{L}_{\mathsf {left}}$ be the path starting from the left end of $\mathsf{L}$ by first proceeding down twice and then alternating between proceeding right and down.  We define $\mathsf{L}_{\mathsf{right}}$ by interchanging left and right in the definition of $\mathsf{L}_{\mathsf{left}}$.
The \emph{staircase} generated by $\mathsf{L}$ is the subgraph of $G_{p,q}$ contained in the region bounded by $\mathsf{L}$, $\mathsf{L}_{\mathsf{left}}$, $\mathsf{L}_{\mathsf {right}}$,  and (possibly) a subpath of the bottommost path of $G_{p,q}$, see Figure~\ref{fig:staircase_a} and  Figure~\ref{fig:staircase_b}.   Staircases for links contained in the leftmost, rightmost, and bottommost paths of $G_{p,q}$ are defined by rotating $G_{p,q}$ so that the link is on the topmost path, applying the above definition, and then rotating back.  

The \emph{length} of a path is its number of edges, and the \emph{length} of a staircase is the length of the link that generates it.  An edge is \emph{thick} if it is covered by at least two bicliques in $\mathcal{C}$. A \emph{pyramid} $\mathsf{P}$ is a staircase of length $2p-4$ containing a size-$2$ fence $\mathsf{B}$ such that the thick edge in $\mathsf{B}$ is the only thick edge in $\mathsf{P}$, see Figure~\ref{fig:pyramid}. We call $\mathsf{B}$ the \emph{tip} of the pyramid. 

\begin{figure}[ht!]
\scalebox{.6}{\begin{tikzpicture}[scale=.8]
\tikzset{frame/.style = {draw=lightgray,line width=.5pt},
biclique/.style = {draw,line width=2pt},
link/.style = {draw=lightblue,line width=2.25pt},
}
    
    \def\e{.2}
    \def\u{.15}
    \def\w{.4}
    \def\p{9}
    \def\q{14}

    \foreach \x in {0,...,\q} {
        \draw [frame] (\x,0) -- (\x,\p);
    }

    \foreach \y in {0,...,\p} {
        \draw [frame] (0,\y) -- (\q,\y);
    }

    \foreach \x in {2,3,4,\q-3,\q-2,\q-1} {
        \begin{scope}[biclique,shift={(\x,\p-1)}]
            \draw[fill=gray!30] (0,0) -- (1,0) -- (1,1) -- (0,1) -- cycle;
        \end{scope}
    }
    
    \draw[link] (5+\u,\p) -- (\q-3-\u,\p);

    \foreach \x in {0,1,2,3,8,9} {
        \begin{scope}[biclique,shift={(\x,0)}]
            \draw[fill=gray!30] (0,0) -- (1,0) -- (1,1) -- (0,1) -- cycle;
        \end{scope}
    }
    
    \draw[link] (4+\u,0) -- (8-\u,0);
    
    \foreach \y in {2,3,6,7} {
        \begin{scope}[biclique,shift={(\q-1,\y)}]
            \draw[fill=gray!30] (0,0) -- (1,0) -- (1,1) -- (0,1) -- cycle;
        \end{scope}
    }
    
    \draw[link] (\q,4+\u) -- (\q,6-\u);
    
    \foreach \y in {3,4} {
        \begin{scope}[biclique,shift={(0,\y)}]
            \draw[fill=gray!30] (0,0) -- (1,0) -- (1,1) -- (0,1) -- cycle;
        \end{scope}
    }
    
    \draw[link] (0,1+\u) -- (0,3-\u);

\end{tikzpicture}}
\caption{Fences are shaded. Links are shown in blue.}
\label{fig:links}
\end{figure}

\begin{figure}[ht!]
\begin{subfigure}[t]{.65\textwidth}\centering
	\scalebox{.6}{\begin{tikzpicture}[scale=1.05]
\tikzset{dot/.style = {draw,fill,circle, inner sep = 2.25pt},
frame/.style = {draw=lightgray,line width=.5pt},
biclique/.style = {draw,line width=2pt},
}
    
    \def\e{.2}
    \def\u{.1}
    \def\w{.4}
    \def\p{5}
    \def\q{13}

    \foreach \x in {0,...,\q} {
        \draw [frame] (\x,0) -- (\x,\p);
    }

    \foreach \y in {0,...,\p} {
        \draw [frame] (-\w,\y) -- (\q+\w,\y);
    }

    \begin{scope}[biclique,shift={(1,\p-1)}]
    \draw[fill=gray!30] (0,0) -- (1,0) -- (1,1) -- (0,1) -- cycle;
    \end{scope}
    \begin{scope}[biclique,shift={(0,\p-1)}]
    \draw[fill=gray!30] (0,0) -- (1,0) -- (1,1) -- (0,1) -- cycle;
    \end{scope}

    \begin{scope}[biclique,shift={(12,\p-1)}]
    \draw[fill=gray!30] (0,0) -- (1,0) -- (1,1) -- (0,1) -- cycle;
    \end{scope}
    
    % \begin{scope}[biclique,shift={(15,\p-1)}]
    % \draw[fill=gray!30] (0,0) -- (1,0) -- (1,1) -- (0,1) -- cycle;
    % \end{scope}
    % \begin{scope}[biclique,shift={(16,\p-1)}]
    % \draw[fill=gray!30] (0,0) -- (1,0) -- (1,1) -- (0,1) -- cycle;
    % \end{scope}

    \draw[lightblue,line width=2pt]
        (2,\p) -- 
        (2,\p-2)--(3,\p-2)--
        (3,\p-3)--(4,\p-3)--
        (4,\p-4)--(5,\p-4)--
        (5,\p-5)--(9,\p-5)--
        (9,\p-4)--(10,\p-4)--
        (10,\p-3)--(11,\p-3)--
        (11,\p-2)--(12,\p-2)--
        (12,\p) -- cycle;
        
    \draw[lightblue,line width=2pt] (3,\p)--(3,\p-2)--(11,\p-2)--(11,\p);
    
    \draw[lightblue,line width=2pt] (4,\p)--(4,\p-3)--(10,\p-3)--(10,\p);
    
    \draw[lightblue,line width=2pt] (5,\p)--(5,\p-4)--(9,\p-4)--(9,\p);
    
    \draw[lightblue,line width=2pt] (2,\p-1)--(12,\p-1);
    
    \draw[lightblue,line width=2pt] (6,\p)--(6,\p-5);
    \draw[lightblue,line width=2pt] (7,\p)--(7,\p-5);
    \draw[lightblue,line width=2pt] (8,\p)--(8,\p-5);
    
\end{tikzpicture}}
	\caption{The staircase generated by a link of length 10 is shown in blue. In this case, the staircase includes part of the bottommost path of the grid.}
	\label{fig:staircase_a}
\end{subfigure}\hspace{3mm}
\begin{subfigure}[t]{.3\textwidth}\centering
	\scalebox{.6}{\begin{tikzpicture}[scale=1.05]
\tikzset{dot/.style = {draw,fill,circle, inner sep = 2.25pt},
frame/.style = {draw=lightgray,line width=.5pt},
biclique/.style = {draw,line width=2pt},
}
    
    \def\e{.2}
    \def\w{.4}
    \def\p{5}
    \def\q{6}

    \foreach \x in {0,...,\q} {
        \draw [frame] (\x,0) -- (\x,\p);
    }

    \foreach \y in {0,...,\p} {
        \draw [frame] (-\w,\y) -- (\q+\w,\y);
    }

    \begin{scope}[biclique,shift={(0,\p-1)}]
    \draw[fill=gray!30] (0,0) -- (1,0) -- (1,1) -- (0,1) -- cycle;
    \end{scope}

    \begin{scope}[biclique,shift={(5,\p-1)}]
    \draw[fill=gray!30] (0,0) -- (1,0) -- (1,1) -- (0,1) -- cycle;
    \end{scope}
    
    \draw[lightblue,line width=2pt]
        (1,\p) -- 
        (1,\p-2)--(2,\p-2)--
        (2,\p-3)--(4,\p-3)--
        (4,\p-2)--(5,\p-2)--
        (5,\p)--cycle;
    
    \draw[lightblue,line width=2pt] (2,\p)--(2,\p-2)--(4,\p-2)--(4,\p);
    
    \draw[lightblue,line width=2pt] (3,\p)--(3,\p-3);
    
    \draw[lightblue,line width=2pt] (1,\p-1)--(5,\p-1);
\end{tikzpicture}}
	\caption{The staircase generated by a link of length 4 is shown in blue.}
	\label{fig:staircase_b}
\end{subfigure}
\caption{}
\label{fig:staircases}
\end{figure}

\begin{figure}[ht!]\centering
\scalebox{.6}{\begin{tikzpicture}[scale=1.05]
\tikzset{dot/.style = {draw,fill,circle, inner sep = 2.25pt},
frame/.style = {draw=lightgray,line width=.5pt},
biclique/.style = {draw,line width=2pt},
}
    
    \def\e{.2}
    \def\u{.1}
    \def\w{.4}
    \def\p{5}
    \def\q{11}
    
    \clip(-\w,-\w) rectangle (\q+\w,\p+\w);

    \foreach \x in {0,...,\q} {
        \draw [frame] (\x,0) -- (\x,\p);
    }

    \foreach \y in {0,...,\p} {
        \draw [frame] (-\w,\y) -- (\q+\w,\y);
    }

    \begin{scope}[biclique,shift={(0,\p-1)}]
    \draw[fill=gray!30] (0,0) -- (1,0) -- (1,1) -- (0,1) -- cycle;
    \end{scope}
    \begin{scope}[biclique,shift={(9,\p-1)}]
    \draw[fill=gray!30] (0,0) -- (1,0) -- (1,1) -- (0,1) -- cycle;
    \end{scope}

    \begin{scope}[biclique,shift={(10,\p-1)}]
    \draw[fill=gray!30] (0,0) -- (1,0) -- (1,1) -- (0,1) -- cycle;
    \end{scope}
    
    % \begin{scope}[biclique,shift={(15,\p-1)}]
    % \draw[fill=gray!30] (0,0) -- (1,0) -- (1,1) -- (0,1) -- cycle;
    % \end{scope}
    % \begin{scope}[biclique,shift={(16,\p-1)}]
    % \draw[fill=gray!30] (0,0) -- (1,0) -- (1,1) -- (0,1) -- cycle;
    % \end{scope}

    \draw[lightblue,line width=2pt]
        (1,\p) -- 
        (1,\p-2)--(2,\p-2)--
        (2,\p-3)--(3,\p-3)--
        (3,\p-4)--(4,\p-4)--
        (4,\p-5)--(6,\p-5)--
        (6,\p-4)--(7,\p-4)--
        (7,\p-3)--(8,\p-3)--
        (8,\p-2)--(9,\p-2)--
        (9,\p) -- cycle;
        
    \draw[lightblue,line width=2pt] (2,\p)--(2,\p-2)--(8,\p-2)--(8,\p);
    
    \draw[lightblue,line width=2pt] (3,\p)--(3,\p-3)--(7,\p-3)--(7,\p);
    
    \draw[lightblue,line width=2pt] (4,\p)--(4,\p-4)--(6,\p-4)--(6,\p);
    
    \draw[lightblue,line width=2pt] (1,\p-1)--(9,\p-1);
    
    \draw[lightblue,line width=2pt] (5,\p)--(5,\p-5);

    \begin{scope}[biclique,shift={(4,0)},darkblue]
    \draw[fill=darkblue!30] (0,0) -- (1,0) -- (1,1) -- (0,1) -- cycle;
    \end{scope}
    
    \begin{scope}[biclique,shift={(5,0)},darkblue]
    \draw[fill=darkblue!30] (0,0) -- (1,0) -- (1,1) -- (0,1) -- cycle;
    \end{scope}
    
    \draw[red,line width=3pt](5,.05)--(5,1-.05);
\end{tikzpicture}}
\caption{A pyramid and its tip (in dark blue). The red edge is the only thick edge.}
\label{fig:pyramid}
\end{figure}
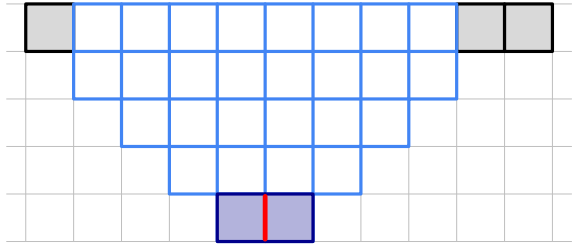

\begin{clm} \label{clm:containingthick}
Every staircase contains at least one thick edge.
\end{clm}

\begin{subproof}
Let $\mathsf{L}$ be a link of length  $k$ connecting two fences $\mathsf{B}$ and $\mathsf{B}'$ (see Figure~\ref{fig:staircase_0}) and let $\mathsf{S}$ be the staircase generated by $\mathsf{L}$.  Towards a contradiction, suppose $\mathsf{S}$ does not contain a thick edge.  Since $\mathcal{C}$ satisfies properties~\ref{tag:item_1} and~\ref{tag:item_2} of Lemma~\ref{lem:cover_props},  $k$ is even and $P$ is covered by exactly $k/2$ boundary stars as in Figure~\ref{fig:staircase_1}. Let $E_1$ be the set of edges in $\mathsf{S}$ with one end on $\mathsf{L}$ and not covered by a boundary element of $\mathcal{C}$ (blue edges in Figure~\ref{fig:staircase_1}).  Each edge $e$ in $E_1$ must be covered by a $K_{1,4}$, else $e$ is thick, see Figure~\ref{fig:staircase_2}.  Let $e_1$ and $e_2$ be the two horizontal edges of $\mathsf{S}$ which intersect $(V(\mathsf{B}) \cup V(\mathsf{B}')) \setminus V(\mathsf{L})$ (blue edges in Figure~\ref{fig:staircase_2}).  Since neither $e_1$ nor $e_2$ are thick, they must be covered by $4$-cycles as in Figure~\ref{fig:staircase_3}.  
By repeating this argument, we either find a thick edge (red edge in Figure~\ref{fig:staircase_4}), or we reach the bottommost path of $G_{p,q}$.  Let $F$ be the set of edges $f \in E(\mathsf S)$ such that $f$ is covered by a $K_{1,4}$ of $\mathcal{C}$ and $f$ has exactly one endpoint on the bottommost path of $G_{p,q}$.  By the above argument, $F$ is non-empty. Moreover, each edge in $F$ is thick (see Figure~\ref{fig:staircase_5}).
\end{subproof}

\begin{figure}[ht!]\centering
\begin{subfigure}{.32\textwidth}\centering
	\scalebox{.69}{\begin{tikzpicture}[scale=.85]
\tikzset{dot/.style = {draw,fill,circle, inner sep = 1.95pt},
frame/.style = {draw=lightgray,line width=.5pt},
biclique/.style = {draw,line width=2pt},
}
    
    \def\e{.2}
    \def\w{.1}
    \def\l{.4}
    \def\p{5}
    \def\q{8}
    
    \clip(-\l,-\l) rectangle (\q+\l,\p+3.2*\l);
    
    \foreach \x in {1,2,3} {
        \fill[gray!10] (\x,\p) -- (\q-\x,\p) -- (\q-\x,\p-\x-1) -- (\x,\p-\x-1) -- cycle;
    }
    
    \foreach \x in {0,...,\q} {
        \draw [frame] (\x,-\l) -- (\x,\p);
    }

    \foreach \y in {0,...,\p} {
        \draw [frame] (-\l,\y) -- (\q+\l,\y);
    }

    \begin{scope}[biclique,shift={(0,\p-1)}]
    \draw[fill=gray!30] (0,0) -- (1,0) -- (1,1) -- (0,1) -- cycle;
    \node[above] at (.5,1) {$\mathsf{B}$};
    \end{scope}

    \begin{scope}[biclique,shift={(\q-1,\p-1)}]
    \draw[fill=gray!30] (0,0) -- (1,0) -- (1,1) -- (0,1) -- cycle;
    \node[above] at (.5,1) {$\mathsf{B}'$};
    \end{scope}
    
    \draw[lightblue,line width=2pt] (1+\w,\p) -- (\q-1-\w,\p) node [midway,above] {$\mathsf{L}$};

\end{tikzpicture}}
	\caption{}
	\label{fig:staircase_0}
\end{subfigure}\hspace{2mm}
\begin{subfigure}{.32\textwidth}\centering
	\scalebox{.69}{\begin{tikzpicture}[scale=.85]
\tikzset{dot/.style = {draw,fill,circle, inner sep = 1.95pt},
frame/.style = {draw=lightgray,line width=.5pt},
biclique/.style = {draw,line width=2pt},
}
    
    \def\e{.2}
    \def\w{.1}
    \def\l{.4}
    \def\p{5}
    \def\q{8}
    
    \clip(-\l,-\l) rectangle (\q+\l,\p+3.2*\l);
    
    \foreach \x in {1,2,3} {
        \fill[gray!10] (\x,\p) -- (\q-\x,\p) -- (\q-\x,\p-\x-1) -- (\x,\p-\x-1) -- cycle;
    }
    
    \foreach \x in {0,...,\q} {
        \draw [frame] (\x,-\l) -- (\x,\p);
    }

    \foreach \y in {0,...,\p} {
        \draw [frame] (-\l,\y) -- (\q+\l,\y);
    }

    \begin{scope}[biclique,shift={(0,\p-1)}]
    \draw[fill=gray!30] (0,0) -- (1,0) -- (1,1) -- (0,1) -- cycle;
    \end{scope}

    \begin{scope}[biclique,shift={(\q-1,\p-1)}]
    \draw[fill=gray!30] (0,0) -- (1,0) -- (1,1) -- (0,1) -- cycle;
    \end{scope}
    
    \foreach \x in {1,2,3} {
        \begin{scope}[biclique,shift={(2*\x,\p)},rotate=180]
        \node[dot] at (0,0) {};
        \draw (-1+\e,0) -- (1-\e,0);
        \draw (0,0) -- (0,1-\e);
        \end{scope}
    }
    
    \foreach \x in {0,2} {
        \draw[line width=2pt,lightblue] (3+\x,\p-1+\w) -- (3+\x,\p-\w);
    }

\end{tikzpicture}}
	\caption{}
	\label{fig:staircase_1}
\end{subfigure}\hspace{2mm}
\begin{subfigure}{.32\textwidth}\centering
	\scalebox{.69}{\begin{tikzpicture}[scale=.85]
\tikzset{dot/.style = {draw,fill,circle, inner sep = 1.95pt},
frame/.style = {draw=lightgray,line width=.5pt},
biclique/.style = {draw,line width=2pt},
}
    
    \def\e{.2}
    \def\w{.1}
    \def\l{.4}
    \def\p{5}
    \def\q{8}
    
    \clip(-\l,-\l) rectangle (\q+\l,\p+3.2*\l);
    
    \foreach \x in {1,2,3} {
        \fill[gray!10] (\x,\p) -- (\q-\x,\p) -- (\q-\x,\p-\x-1) -- (\x,\p-\x-1) -- cycle;
    }
    
    \foreach \x in {0,...,\q} {
        \draw [frame] (\x,-\l) -- (\x,\p);
    }

    \foreach \y in {0,...,\p} {
        \draw [frame] (-\l,\y) -- (\q+\l,\y);
    }

    \begin{scope}[biclique,shift={(0,\p-1)}]
    \draw[fill=gray!30] (0,0) -- (1,0) -- (1,1) -- (0,1) -- cycle;
    \end{scope}

    \begin{scope}[biclique,shift={(\q-1,\p-1)}]
    \draw[fill=gray!30] (0,0) -- (1,0) -- (1,1) -- (0,1) -- cycle;
    \end{scope}
    
    \foreach \x in {1,2,3} {
        \begin{scope}[biclique,shift={(2*\x,\p)},rotate=180]
        \node[dot] at (0,0) {};
        \draw (-1+\e,0) -- (1-\e,0);
        \draw (0,0) -- (0,1-\e);
        \end{scope}
    }
    
    \foreach \x in {0,2} {
        \begin{scope}[biclique,shift={(3+\x,\p-1)}]
        \node[dot] at (0,0) {};
        \draw (-1+\e,0) -- (1-\e,0);
        \draw (0,-1+\e) -- (0,1-\e);
        \end{scope}
    }
    
    \foreach \x in {0,1} {
        \draw[line width=2pt,lightblue] (5*\x+1+\w,\p-1) -- (5*\x+2-\w,\p-1);
    }

\end{tikzpicture}}
	\caption{}
	\label{fig:staircase_2}
\end{subfigure}\\[3mm]
\begin{subfigure}{.32\textwidth}\centering
	\scalebox{.69}{\begin{tikzpicture}[scale=.85]
\tikzset{dot/.style = {draw,fill,circle, inner sep = 1.95pt},
frame/.style = {draw=lightgray,line width=.5pt},
biclique/.style = {draw,line width=2pt},
}
    
    \def\e{.2}
    \def\w{.1}
    \def\l{.4}
    \def\p{5}
    \def\q{8}
    
    \clip(-\l,-\l) rectangle (\q+\l,\p+3.2*\l);
    
    \foreach \x in {1,2,3} {
        \fill[gray!10] (\x,\p) -- (\q-\x,\p) -- (\q-\x,\p-\x-1) -- (\x,\p-\x-1) -- cycle;
    }
    
    \foreach \x in {0,...,\q} {
        \draw [frame] (\x,-\l) -- (\x,\p);
    }

    \foreach \y in {0,...,\p} {
        \draw [frame] (-\l,\y) -- (\q+\l,\y);
    }

    \begin{scope}[biclique,shift={(0,\p-1)}]
    \draw[fill=gray!30] (0,0) -- (1,0) -- (1,1) -- (0,1) -- cycle;
    \end{scope}

    \begin{scope}[biclique,shift={(\q-1,\p-1)}]
    \draw[fill=gray!30] (0,0) -- (1,0) -- (1,1) -- (0,1) -- cycle;
    \end{scope}
    
    \foreach \x in {1,2,3} {
        \begin{scope}[biclique,shift={(2*\x,\p)},rotate=180]
        \node[dot] at (0,0) {};
        \draw (-1+\e,0) -- (1-\e,0);
        \draw (0,0) -- (0,1-\e);
        \end{scope}
    }
    
    \foreach \x in {0,2} {
        \begin{scope}[biclique,shift={(3+\x,\p-1)}]
        \node[dot] at (0,0) {};
        \draw (-1+\e,0) -- (1-\e,0);
        \draw (0,-1+\e) -- (0,1-\e);
        \end{scope}
    }
    
    \foreach \x in {1,6} {
        \begin{scope}[biclique,shift={(\x,\p-2)}]
            \draw[fill=gray!30] (0,0) -- (1,0) -- (1,1) -- (0,1) -- cycle;
        \end{scope}
    }

\end{tikzpicture}}
	\caption{}
	\label{fig:staircase_3}
\end{subfigure}\hspace{2mm}
\begin{subfigure}{.32\textwidth}\centering
	\scalebox{.69}{\begin{tikzpicture}[scale=.85]
\tikzset{dot/.style = {draw,fill,circle, inner sep = 1.95pt},
frame/.style = {draw=lightgray,line width=.5pt},
biclique/.style = {draw,line width=2pt},
}
    
    \def\e{.2}
    \def\w{.1}
    \def\l{.4}
    \def\p{5}
    \def\q{8}
    
    \clip(-\l,-\l) rectangle (\q+\l,\p+3.2*\l);
    
    \foreach \x in {1,2,3} {
        \fill[gray!10] (\x,\p) -- (\q-\x,\p) -- (\q-\x,\p-\x-1) -- (\x,\p-\x-1) -- cycle;
    }
    
    \foreach \x in {0,...,\q} {
        \draw [frame] (\x,-\l) -- (\x,\p);
    }

    \foreach \y in {0,...,\p} {
        \draw [frame] (-\l,\y) -- (\q+\l,\y);
    }

    \begin{scope}[biclique,shift={(0,\p-1)}]
    \draw[fill=gray!30] (0,0) -- (1,0) -- (1,1) -- (0,1) -- cycle;
    \end{scope}

    \begin{scope}[biclique,shift={(\q-1,\p-1)}]
    \draw[fill=gray!30] (0,0) -- (1,0) -- (1,1) -- (0,1) -- cycle;
    \end{scope}
    
    \foreach \x in {1,2,3} {
        \begin{scope}[biclique,shift={(2*\x,\p)},rotate=180]
        \node[dot] at (0,0) {};
        \draw (-1+\e,0) -- (1-\e,0);
        \draw (0,0) -- (0,1-\e);
        \end{scope}
    }
    
    \foreach \x in {0,2} {
        \begin{scope}[biclique,shift={(3+\x,\p-1)}]
        \node[dot] at (0,0) {};
        \draw (-1+\e,0) -- (1-\e,0);
        \draw (0,-1+\e) -- (0,1-\e);
        \end{scope}
    }
    
    \begin{scope}[biclique,shift={(4,\p-2)}]
    \node[dot] at (0,0) {};
    \draw (-1+\e,0) -- (1-\e,0);
    \draw (0,-1+\e) -- (0,1-\e);
    \end{scope}

    \foreach \x in {1,6} {
        \begin{scope}[biclique,shift={(\x,\p-2)}]
            \draw[fill=gray!30] (0,0) -- (1,0) -- (1,1) -- (0,1) -- cycle;
        \end{scope}
    }
    
    \foreach \x in {2,5} {
        \begin{scope}[biclique,shift={(\x,\p-3)}]
            \draw[fill=gray!30] (0,0) -- (1,0) -- (1,1) -- (0,1) -- cycle;
        \end{scope}
    }
    
    \foreach \x in {3,4} {
        \begin{scope}[biclique,shift={(\x,\p-4)}]
            \draw[fill=gray!30] (0,0) -- (1,0) -- (1,1) -- (0,1) -- cycle;
        \end{scope}
    }

    \draw[line width=3pt,red] (4,1+\w) -- (4,2-\w);
   
\end{tikzpicture}}
	\caption{}
	\label{fig:staircase_4}
\end{subfigure}\hspace{2mm}
\begin{subfigure}{.32\textwidth}\centering
	\scalebox{.69}{\begin{tikzpicture}[scale=.85]
\tikzset{dot/.style = {draw,fill,circle, inner sep = 1.95pt},
frame/.style = {draw=lightgray,line width=.5pt},
biclique/.style = {draw,line width=2pt},
}
    
    \def\e{.2}
    \def\w{.1}
    \def\l{.4}
    \def\p{2}
    \def\q{8}
    
    \clip(-\l,-1.5-\l) rectangle (\q+\l,\p+1.5+3.2*\l);
    
    \fill[gray!10] (1,\p) -- (\q-1,\p) -- (\q-1,0) -- (1,0) -- cycle;
    
    \foreach \x in {0,...,\q} {
        \draw [frame] (\x,0) -- (\x,\p+\l);
    }

    \foreach \y in {0,...,\p} {
        \draw [frame] (-\l,\y) -- (\q+\l,\y);
    }

    \begin{scope}[biclique,shift={(0,\p-1)}]
    \draw[fill=gray!30] (0,0) -- (1,0) -- (1,1) -- (0,1) -- cycle;
    \end{scope}

    % \begin{scope}[biclique,shift={(-1,\p)}]
    % \draw[opacity=.2,fill=gray!30] (0,0) -- (1,0) -- (1,1) -- (0,1) -- cycle;
    % \end{scope}
    
    % \begin{scope}[biclique,shift={(\q,\p)}]
    % \draw[opacity=.2,fill=gray!30] (0,0) -- (1,0) -- (1,1) -- (0,1) -- cycle;
    % \end{scope}
    
    \begin{scope}[biclique,shift={(\q-1,\p-1)}]
    \draw[fill=gray!30] (0,0) -- (1,0) -- (1,1) -- (0,1) -- cycle;
    \end{scope}
    
    \foreach \x in {1,2,3} {
        \begin{scope}[biclique,shift={(2*\x,\p)}]
        \node[dot] at (0,0) {};
        \draw (-1+\e,0) -- (1-\e,0);
        \draw (0,-1+\e) -- (0,\l);
        \end{scope}
    }

    \foreach \x in {1,6} {
        \begin{scope}[biclique,shift={(\x,\p-2)}]
            \draw[fill=gray!30] (0,0) -- (1,0) -- (1,1) -- (0,1) -- cycle;
        \end{scope}
    }

    \draw[line width=3pt,red] (3,\w) -- (3,1-\w);
    
    \draw[line width=3pt,red] (5,\w) -- (5,1-\w);
    
    \foreach \x in {3,5} {
        \begin{scope}[biclique,shift={(\x,\p-1)}]
        \node[dot] at (0,0) {};
        \draw (-1+\e,0) -- (1-\e,0);
        \draw (0,0) -- (0,1-\e);
        \end{scope}
    }
   
\end{tikzpicture}}
	\caption{}
	\label{fig:staircase_5}
\end{subfigure}
\caption{}
\label{fig:staircase_thick}
\end{figure}

A \emph{double staircase} is a pair $(\mathsf{S}, \mathsf{S}')$, where $\mathsf{S}$ and $\mathsf{S}'$ are distinct staircases such that there is exactly one thick edge in $\mathsf{S} \cup \mathsf{S}'$. Suppose $\mathsf{S}$ has length $2a$ and  $\mathsf{S}'$ has length $2b$. Observe that either $a+b=p-2$ or $a+b=p$ and $\mathcal C$ must cover $\mathsf{S} \cup \mathsf{S}'$ as in Figures~\ref{fig:double_staircase_0} and~\ref{fig:double_staircase_2}, respectively.  

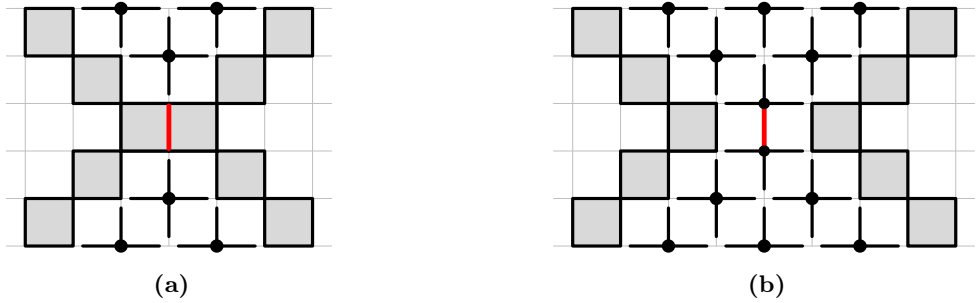
\begin{figure}[ht!]\centering
\begin{subfigure}{.45\textwidth}\centering
    \scalebox{.6}{\begin{tikzpicture}[scale=1.05]
\tikzset{dot/.style = {draw,fill,circle, inner sep = 2.25pt},
frame/.style = {draw=lightgray,line width=.5pt},
biclique/.style = {draw,line width=2pt},
}
    
    \def\e{.2}
    \def\w{.4}
    \def\p{5}
    \def\q{6}
    
    \clip (-\w,-\e) rectangle (\q+\w,\p+\e);

    \foreach \x in {0,...,\q} {
        \draw [frame] (\x,0) -- (\x,\p);
    }

    \foreach \y in {0,...,\p} {
        \draw [frame] (-\w,\y) -- (\q+\w,\y);
    }

    \foreach \x in {1,2,3} {
        \begin{scope}[biclique,shift={(\x-1,\p-\x)}]
            \draw[fill=gray!30] (0,0) -- (1,0) -- (1,1) -- (0,1) -- cycle;
        \end{scope}
        
        \begin{scope}[biclique,shift={(\q-\x,\x-1)}]
           \draw[fill=gray!30] (0,0) -- (1,0) -- (1,1) -- (0,1) -- cycle;
        \end{scope}
        
        \begin{scope}[biclique,shift={(\x-1,\x-1)}]
            \draw[fill=gray!30] (0,0) -- (1,0) -- (1,1) -- (0,1) -- cycle;
        \end{scope}
    
        \begin{scope}[biclique,shift={(\q-\x,\p-\x)}]
            \draw[fill=gray!30] (0,0) -- (1,0) -- (1,1) -- (0,1) -- cycle;
        \end{scope}
    }

    \begin{scope}[biclique,shift={(3,\p-1)}]
        \node [dot] at (0,0) {}; 
        \draw (-1+\e,0) -- (1-\e,0);
        \draw (0,-1+\e) -- (0,1-\e);
    \end{scope}
    
    \begin{scope}[biclique,shift={(3,1)}]
        \node [dot] at (0,0) {}; 
        \draw (-1+\e,0) -- (1-\e,0);
        \draw (0,-1+\e) -- (0,1-\e);
    \end{scope}
    
    \begin{scope}[biclique,shift={(2,0)}]
        \node [dot] at (0,0) {}; 
        \draw (-1+\e,0) -- (1-\e,0);
        \draw (0,0) -- (0,1-\e);
    \end{scope}
    
    \begin{scope}[biclique,shift={(4,0)}]
        \node [dot] at (0,0) {}; 
        \draw (-1+\e,0) -- (1-\e,0);
        \draw (0,0) -- (0,1-\e);
    \end{scope}
    
    \begin{scope}[biclique,shift={(2,\p)}]
        \node [dot] at (0,0) {}; 
        \draw (-1+\e,0) -- (1-\e,0);
        \draw (0,-1+\e) -- (0,0);
    \end{scope}
    
    \begin{scope}[biclique,shift={(4,\p)}]
        \node [dot] at (0,0) {}; 
        \draw (-1+\e,0) -- (1-\e,0);
        \draw (0,-1+\e) -- (0,0);
    \end{scope}
    
    \draw[red,line width=3pt](3,2+.05)--(3,3-.05);

\end{tikzpicture}}
    \caption{}
    \label{fig:double_staircase_0}
\end{subfigure}\hspace{3mm}
\begin{subfigure}{.45\textwidth}\centering
	\scalebox{.6}{\begin{tikzpicture}[scale=1.05]
\tikzset{dot/.style = {draw,fill,circle, inner sep = 2.25pt},
frame/.style = {draw=lightgray,line width=.5pt},
biclique/.style = {draw,line width=2pt},
}
    
    \def\e{.2}
    \def\w{.4}
    \def\p{5}
    \def\q{8}
    
    \clip (-\w,-\e) rectangle (\q+\w,\p+\e);

    \foreach \x in {0,...,\q} {
        \draw [frame] (\x,0) -- (\x,\p);
    }

    \foreach \y in {0,...,\p} {
        \draw [frame] (-\w,\y) -- (\q+\w,\y);
    }

    \foreach \x in {1,2,3} {
        \begin{scope}[biclique,shift={(\x-1,\p-\x)}]
            \draw[fill=gray!30] (0,0) -- (1,0) -- (1,1) -- (0,1) -- cycle;
        \end{scope}
        
        \begin{scope}[biclique,shift={(\q-\x,\x-1)}]
           \draw[fill=gray!30] (0,0) -- (1,0) -- (1,1) -- (0,1) -- cycle;
        \end{scope}
        
        \begin{scope}[biclique,shift={(\x-1,\x-1)}]
            \draw[fill=gray!30] (0,0) -- (1,0) -- (1,1) -- (0,1) -- cycle;
        \end{scope}
    
        \begin{scope}[biclique,shift={(\q-\x,\p-\x)}]
            \draw[fill=gray!30] (0,0) -- (1,0) -- (1,1) -- (0,1) -- cycle;
        \end{scope}
    }
    
    \foreach \x in {3,5} {
        \begin{scope}[biclique,shift={(\x,\p-1)}]
            \node [dot] at (0,0) {}; 
            \draw (-1+\e,0) -- (1-\e,0);
            \draw (0,-1+\e) -- (0,1-\e);
        \end{scope}
        
        \begin{scope}[biclique,shift={(\x,1)}]
            \node [dot] at (0,0) {}; 
            \draw (-1+\e,0) -- (1-\e,0);
            \draw (0,-1+\e) -- (0,1-\e);
        \end{scope}
    }

    \foreach \x in {2,4,6} {
        \begin{scope}[biclique,shift={(\x,0)}]
            \node [dot] at (0,0) {}; 
            \draw (-1+\e,0) -- (1-\e,0);
            \draw (0,0) -- (0,1-\e);
        \end{scope}
        
        \begin{scope}[biclique,shift={(\x,\p)}]
            \node [dot] at (0,0) {}; 
            \draw (-1+\e,0) -- (1-\e,0);
            \draw (0,-1+\e) -- (0,0);
        \end{scope}
    }
    
    \begin{scope}[biclique,shift={(4,2)}]
        \draw (-1+\e,0) -- (1-\e,0);
        \draw (0,-1+\e) -- (0,1-\e);
    \end{scope}
    
    \begin{scope}[biclique,shift={(4,3)}]
        \draw (-1+\e,0) -- (1-\e,0);
        \draw (0,-1+\e) -- (0,1-\e);
    \end{scope}
    
    \draw[red,line width=3pt](4,2+.1)--(4,3-.1);
    
    \node [dot] at (4,2) {}; 
    \node [dot] at (4,3) {}; 
    
\end{tikzpicture}}
    \caption{}
    \label{fig:double_staircase_2}
\end{subfigure}
\caption{Examples of double staircases. The red edge is the only thick edge.}
\label{fig:double_staircase}
\end{figure}

We define the \emph{waste} of $\mathcal{C}$ to be 
\begin{equation} \label{eq:def_waste1}
w(\mathcal{C}) \coloneqq 4|\mathcal{C}| - |E(G_{p,q})|. 
\end{equation}
Let $\tau$ be the number of copies of $K_{1,3}$ that appear in $\mathcal{C}$ and $t_i$ be the number of edges of $G_{p,q}$ that are covered exactly $i$ times by $\mathcal{C}$. By double counting, we obtain:
\begin{equation} \label{eq:def_waste2}
w(\mathcal{C})=\tau+t_2+2t_3+3t_4.
\end{equation}

Since $|\mathcal{C}|=\nicefrac{pq}{2}-1$ and $|E(G_{p,q})|=2pq-p-q$,~\eqref{eq:def_waste1} immediately gives the following:
\begin{equation}\label{eq:minimal_waste}
w(\mathcal{C})=p+q-4\,.
\end{equation}

Note that a fence can have size at most $N \coloneqq 2p+2q-8$.  For $i \in [N]$, let $b_i$ be the number of fences of size $i$ of $\mathcal{C}$.

\begin{clm}\label{clm:trivialities}
$b_N=0$.
\end{clm}

\begin{subproof}
If $b_N=1$, then $\mathcal C$ contains $2p+2q-8$ boundary $4$-cycles.  These boundary $4$-cycles cover $6p+6q-24$ edges of $G_{p,q}$.  Thus, $2pq-7p-7q+24$ edges are not covered by boundary $4$-cycles.  It follows that 
\[
|\mathcal C| \geq 2p+2q-8+\frac{2pq-7p-7q+24}{4}=\frac{pq}{2}+\frac{p}{4}+\frac{q}{4}-2 \geq \frac{pq}{2},
\]
where the last inequality follows since $4 \leq p \leq q$.  However, this contradicts $|\mathcal C|=\nicefrac{pq}{2}-1$.  
\end{subproof}

Let $\beta$ be the number of edges that are covered twice by boundary $4$-cycles of $\mathcal{C}$, and let $c \in [4]$ be the number of corners of $G_{p,q}$ covered by the set of fences.  We now provide a lower bound on the waste of $\mathcal{C}$.

\begin{clm} \label{clm:wastelowerbound}
$w(\mathcal{C}) \geq\displaystyle \beta+\tau = p+q-2-\frac{c}{2} -\frac{b_1}{2} +\sum_{i = 3}^{N-1} \Big(\frac{i}{2}-1\Big)b_i$.  
\end{clm}

\begin{subproof}
The inequality follows immediately from~\eqref{eq:def_waste2}.  For the equality, first observe 
\(
\beta = \sum_{i=1}^{N-1} (i-1)b_i
\).
Next, because each boundary star covers two edges of $\out$, each edge of $\out$ not covered by a fence contributes $\nicefrac{1}{2}$ to $\tau$. Since $\out$ contains $2p+2q-4$ edges and the set of fences cover $c+\sum_{i=1}^{N-1} ib_i$ edges of $\out$, we have
\begin{align*}
\beta+\tau &= \sum_{i = 1}^{N-1} (i-1)b_i + \frac{1}{2}\bigg(2p+2q-4-c - \sum_{i = 1}^{N-1} ib_i\bigg)  \\[-2pt]
&=p+q-2-\frac{c}{2}- \frac{b_1}{2} +\sum_{i = 3}^{N-1} \Big(\frac{i}{2}-1\Big)b_i. \qedhere
\end{align*}
\end{subproof}

\begin{clm} \label{clm:partition}
The set of staircases of $\mathcal{C}$ can be enumerated as $\mathsf{P}_1, \dots , \mathsf{P}_n, \mathsf{S}_1, \mathsf{S}_1', \dots, \mathsf{S}_m, \mathsf{S}_m'$ such that $\mathsf{P}_i$ is a pyramid for all $i \in [n]$ and $(\mathsf{S}_j, \mathsf{S}_j')$ is a double staircase for all $j \in [m]$.  Moreover, $b_2=n$, $b_i=0$ for all $i \geq 3$, and $w(\mathcal{C})=\tau+n+m$.
\end{clm}

\begin{subproof}
The number of staircases of $\mathcal{C}$ is at least $\sum_{i=1}^{N-1}b_i+c-4$.  At most $b_2$ fences can be the tips of pyramids, so at least $b_1+\sum_{i=3}^{N-1}b_i+c-4$ staircases are not pyramids. Each of these staircases contains a thick edge by Claim~\ref{clm:containingthick}, and each thick edge can be in at most $2$ staircases.  Therefore, these staircases contribute at least $(b_1+\sum_{i=3}^{N-1}b_i+c-4)/2$ to the waste of $\mathcal{C}$ which is not counted in $\beta+\tau$.  By Claim~\ref{clm:wastelowerbound},
\begin{align*}
w(\mathcal{C})  &\geq p+q-2-\frac{c}{2} -\frac{b_1}{2}+\sum_{i = 3}^{N-1} \Big(\frac{i}{2}-1\Big)b_i + \frac{1}{2}\bigg(b_1+\sum_{i=3}^{N-1}b_i+c-4\bigg) \\[-2pt]
 &=p+q-4+\sum_{i=3}^{N-1}\Big(\frac{i-1}{2}\Big)b_i.%+\frac{N+1}{2}b_N. 
\end{align*}
By~\eqref{eq:minimal_waste}, $w(\mathcal{C})=p+q-4$. Hence, we must have $b_i=0$ for $i \ge 3$ and equality throughout the above argument. In particular, there are exactly $b_2$ pyramids.  The remaining staircases each contain exactly one thick edge, and each of these thick edges is in exactly two staircases.  This gives the required enumeration of the staircases of $\mathcal{C}$ and implies $w(\mathcal{C})=\tau+n+m$.
\end{subproof}

\begin{clm}
$q-1 = k (p-1) + 2 \ell$ for some integers $0 \le \ell < k $.
\end{clm}

\begin{subproof}
Let $L$ be the leftmost path of $G_{p,q}$.  Since $q \geq p$, no link contained in $L$ can generate a pyramid or be part of a double staircase.  By Claim~\ref{clm:partition}, no link is contained in $L$.  It follows that $L$ intersects at most one fence.  Since $p$ is even, it follows that $L$ intersects exactly one fence $\mathsf{B}_{\mathsf{left}}$.  Since all fences are of size $1$ or $2$ and $p$ is even,  $\mathsf{B}_{\mathsf{left}}$ is of size $1$ and $L \setminus E(\mathsf{B}_{\mathsf{left}})$ is the disjoint union of two even-length paths $P_a$ and $P_b$.  We assume that $P_a$ is above $P_b$ and that $P_a$ and $P_b$ have lengths $a$ and $b$, respectively. By the same argument, the rightmost path of $G_{p,q}$ intersects exactly one fence $\mathsf{B}_{\mathsf{right}}$.  

%  If the set of staircases is empty, then no edge of $G_{p,q}$ is thick by Claim~\ref{clm:partition}.  Therefore, $\mathsf{B}_{\mathsf{left}}$ and $\mathsf{B}_{\mathsf{right}}$ force $G_{p,q}$ to be covered by $\mathcal C$ as in Figure~\ref{fig:grid-6_6}.  In particular, $q=p$, so we may take $k=1$ and $\ell=0$.  
% By symmetry, we may assume that there is a link on the topmost path of $G_{p,q}$ and we 

We first consider the case when $\mathsf{B}_{\mathsf{left}}$ is a fence that does not contain a corner.
Suppose $\mathsf{L}_1$ is the first link along the topmost path of $G_{p,q}$ and $\mathsf{S}_1$ is the staircase generated by $\mathsf{L}_1$.  By Claim~\ref{clm:partition}, there are no thick edges outside of staircases. Therefore, $\mathsf{B}_{\mathsf{left}}$ and the boundary stars force the leftmost portion of $G_{p,q}$ to be covered by $\mathcal{C}$ as in Figure~\ref{fig:grid-10_0}.  This forces $\mathsf{S}_1$ to be in a double staircase $(\mathsf{S}_1, \mathsf{S}_1')$, where $\mathsf{S}_1$ has length $2b$ and $\mathsf{S}_1'$ has length $2a$ (see Figures~\ref{fig:grid-10_2} and~\ref{fig:grid-10_1}), or $\mathsf{S}_1$ has length $2b+2$ and $\mathsf{S}_1'$ has length $2a+2$ (see Figures~\ref{fig:grid-10_4} and~\ref{fig:grid-10_3}). 

Suppose $\mathsf{L}_2$ is the next link along the topmost path of $G_{p,q}$. Because of the previous double staircase $(\mathsf{S}_1, \mathsf{S}_1')$, $\mathsf{S}_2$ must be in a double staircase $(\mathsf{S}_2, \mathsf{S}_2')$,  where $\mathsf{S}_2$ has length $2a$ and $\mathsf{S}_2'$ has length $2b$ (see Figures~\ref{fig:grid-10_2} and~\ref{fig:grid-10_4}), or $\mathsf{S}_2$ has length $2a+2$ and $\mathsf{S}_2'$ has length $2b+2$ (see Figures~\ref{fig:grid-10_1} and~\ref{fig:grid-10_3}).  Repeating the argument, we obtain a sequence of double staircases $(\mathsf{S}_1, \mathsf{S}_1'), \dots, (\mathsf{S}_{k-1}, \mathsf{S}_{k-1}')$, where $(\mathsf{S}_{k-1}, \mathsf{S}_{k-1}')$ is the last double staircase.  Note that it is possible that this sequence is empty, corresponding to the case that there are no staircases and $k-1=0$.

Again, because there are no thick edges outside of staircases, the fence $\mathsf{B}_{\mathsf{right}}$ and the boundary stars force the rightmost portion of $G_{p,q}$ to be covered as in Figure~\ref{fig:grid-10_0_mirror}.  For all $i \in [k-1]$, let $\ell_i$ and $\ell_i'$ be the lengths of $\mathsf{S}_i$ and $\mathsf{S}_i'$ respectively.  Let $\ell$ be the number of times that $\{\ell_i,\ell_i'\}=\{2a+2, 2b+2\}$.  It follows that $q-1 = k (p-1) + 2 \ell$, as required.

The case when $\mathsf{B}_{\mathsf{left}}$ contains a corner is treated exactly as above, resulting in the construction from Figure~\ref{fig:grid_example_2}.
\end{subproof}
%
% \begin{figure}[ht!]\centering
% \scalebox{.7}{\input{grid-6_6.tex}}
% \caption{}
% \label{fig:grid-6_6}
% \end{figure}
%
\begin{figure}[ht!]\centering
\begin{subfigure}{.48\textwidth}\centering
	\scalebox{.55}{\begin{tikzpicture}[scale=.8]
\tikzset{dot/.style = {draw,fill,circle, inner sep = 1.9pt},
frame/.style = {draw=lightgray,line width=.5pt},
biclique/.style = {draw,line width=2pt},
brace/.style = {decorate,decoration={brace,amplitude=10pt},xshift=0pt,yshift=0pt,line width=1.5pt}
}
    \def\e{.2}
    \def\p{9}
    \def\q{16}
    \def\w{.1}
    \def\h{.3}

    \foreach \x in {0,...,\q} {
        \draw [frame] (\x,0) -- (\x,\p);
    }

    \foreach \y in {0,...,\p} {
        \draw [frame] (0,\y) -- (\q+.4,\y);
    }

\node[left] at (0,6.5) {\color{darkblue}\Large$\mathsf{B}_{\mathsf{left}}$};

\begin{scope}[]
    
    \foreach \x in {0,...,6} {
        \begin{scope}[biclique,shift={(\x+2,\p-\x-1)}]
        \draw[fill=gray!30] (0,0) -- (1,0) -- (1,1) -- (0,1) -- cycle;
        \end{scope}
    }
    
    \foreach \x in {0,1} {
        \begin{scope}[biclique,shift={(\p-2-\x,1-\x)}]
        \draw[fill=gray!30] (0,0) -- (1,0) -- (1,1) -- (0,1) -- cycle;
        \end{scope}
    }

    \foreach \x in {0,...,5} {
        \foreach \y in {1,2} {
             \begin{scope}[lightblue,biclique,shift={(\x-\y+2,\p-\x-1-\y)}]
                \draw[fill=lightblue!30] (0,0) -- (1,0) -- (1,1) -- (0,1) -- cycle;
            \end{scope}
        }
    }
    
    \begin{scope}[lightblue] 
        \begin{scope}[biclique,shift={(0,\p-1)}]
        \node[dot] at (0,0) {};
        \draw (0,0) -- (1-\e,0);
        \draw (0,-1+\e) -- (0,1-\e);
        \end{scope}
        
        \begin{scope}[biclique,shift={(1,\p)}]
        \node[dot] at (0,0) {};
        \draw (-1+\e,0) -- (1-\e,0);
        \draw (0,-1+\e) -- (0,0);
        \end{scope}
        
        %%%%% K_{1,3}
        \foreach \x in {1,3,5} {
            \begin{scope}[biclique,shift={(\x,0)},rotate=0]
            \node[dot] at (0,0) {};
            \draw (-1+\e,0) -- (1-\e,0);
            \draw (0,0) -- (0,1-\e);
            \end{scope}
        }
        
        \foreach \y in {1,3,5} {
            \begin{scope}[biclique,shift={(0,\y)},rotate=-90]
            \node[dot] at (0,0) {};
            \draw (-1+\e,0) -- (1-\e,0);
            \draw (0,0) -- (0,1-\e);
            \end{scope}
        }
        
        %%%%% K_{1,4}
        \foreach \x in {0,1} {
            \begin{scope}[biclique,shift={(1+\x,2-\x)}]
            \node[dot] at (0,0) {};
            \draw (-1+\e,0) -- (1-\e,0);
            \draw (0,-1+\e) -- (0,1-\e);
            \end{scope}
        }
        
        \foreach \x in {0,...,3} {
            \begin{scope}[biclique,shift={(1+\x,4-\x)}]
            \node[dot] at (0,0) {};
            \draw (-1+\e,0) -- (1-\e,0);
            \draw (0,-1+\e) -- (0,1-\e);
            \end{scope}
        }
    \end{scope}
    
    \begin{scope}[biclique,darkblue,shift={(0,6)}]
    \draw[fill=darkblue!30] (0,0) -- (1,0) -- (1,1) -- (0,1) -- cycle;
    \end{scope}
    
    \begin{scope}[shift={(\p,\p)},xscale=-1,yscale=-1]
        
    \begin{scope}[biclique,shift={(0,\p-1)}]
    \node[dot] at (0,0) {};
    \draw (0,0) -- (1-\e,0);
    \draw (0,-1+\e) -- (0,1-\e);
    \end{scope}
    
    \begin{scope}[biclique,shift={(1,\p)}]
    \node[dot] at (0,0) {};
    \draw (-1+\e,0) -- (1-\e,0);
    \draw (0,-1+\e) -- (0,0);
    \end{scope}
    
    %%%%% K_{1,3}

    \foreach \x in {1,3,5} {
        \begin{scope}[biclique,shift={(\x,0)},rotate=0]
        \node[dot] at (0,0) {};
        \draw (-1+\e,0) -- (1-\e,0);
        \draw (0,0) -- (0,1-\e);
        \end{scope}
    }
    
    \foreach \y in {1,3,5} {
        \begin{scope}[biclique,shift={(0,\y)},rotate=-90]
        \node[dot] at (0,0) {};
        \draw (-1+\e,0) -- (1-\e,0);
        \draw (0,0) -- (0,1-\e);
        \end{scope}
    }
    
    %%%%% K_{1,4}
    \foreach \x in {0,1} {
        \begin{scope}[biclique,shift={(1+\x,2-\x)}]
        \node[dot] at (0,0) {};
        \draw (-1+\e,0) -- (1-\e,0);
        \draw (0,-1+\e) -- (0,1-\e);
        \end{scope}
    }
    
    \foreach \x in {0,...,3} {
        \begin{scope}[biclique,shift={(1+\x,4-\x)}]
        \node[dot] at (0,0) {};
        \draw (-1+\e,0) -- (1-\e,0);
        \draw (0,-1+\e) -- (0,1-\e);
        \end{scope}
    }
    
    \end{scope}

\end{scope}

\begin{scope}[shift={(18,0)},xscale=-1]

    \foreach \x in {0,...,6} {
        \begin{scope}[biclique,shift={(\x+2,\p-\x-1)}]
        \draw[fill=gray!30] (0,0) -- (1,0) -- (1,1) -- (0,1) -- cycle;
        \end{scope}
    }
    
    \foreach \x in {0,1} {
        \begin{scope}[biclique,shift={(\p-2-\x,1-\x)}]
        \draw[fill=gray!30] (0,0) -- (1,0) -- (1,1) -- (0,1) -- cycle;
        \end{scope}
    }
    
    \begin{scope}[shift={(\p,\p)},xscale=-1,yscale=-1]
        
    \begin{scope}[biclique,shift={(0,\p-1)}]
    \node[dot] at (0,0) {};
    \draw (0,0) -- (1-\e,0);
    \draw (0,-1+\e) -- (0,1-\e);
    \end{scope}
    
    \begin{scope}[biclique,shift={(1,\p)}]
    \node[dot] at (0,0) {};
    \draw (-1+\e,0) -- (1-\e,0);
    \draw (0,-1+\e) -- (0,0);
    \end{scope}
    
    %%%%% K_{1,3}

    \foreach \x in {1,3,5} {
        \begin{scope}[biclique,shift={(\x,0)},rotate=0]
        \node[dot] at (0,0) {};
        \draw (-1+\e,0) -- (1-\e,0);
        \draw (0,0) -- (0,1-\e);
        \end{scope}
    }
    
    \foreach \y in {1,3,5} {
        \begin{scope}[biclique,shift={(0,\y)},rotate=-90]
        \node[dot] at (0,0) {};
        \draw (-1+\e,0) -- (1-\e,0);
        \draw (0,0) -- (0,1-\e);
        \end{scope}
    }
    
    %%%%% K_{1,4}
    \foreach \x in {0,1} {
        \begin{scope}[biclique,shift={(1+\x,2-\x)}]
        \node[dot] at (0,0) {};
        \draw (-1+\e,0) -- (1-\e,0);
        \draw (0,-1+\e) -- (0,1-\e);
        \end{scope}
    }
    
    \foreach \x in {0,...,3} {
        \begin{scope}[biclique,shift={(1+\x,4-\x)}]
        \node[dot] at (0,0) {};
        \draw (-1+\e,0) -- (1-\e,0);
        \draw (0,-1+\e) -- (0,1-\e);
        \end{scope}
    }
    
    \end{scope}
    
    %\draw [brace] (\p-\w,-\h) -- (0+\w,-\h) node [lightblue,midway,yshift=-8mm] {$G_{8,8}$};
\end{scope}

\end{tikzpicture}}
    \caption{}
    \label{fig:grid-10_0}
\end{subfigure}\hspace{5mm}
\begin{subfigure}{.48\textwidth}\centering
	\scalebox{.55}{\begin{tikzpicture}[scale=.8]
\tikzset{dot/.style = {draw,fill,circle, inner sep = 1.9pt},
frame/.style = {draw=lightgray,line width=.5pt},
biclique/.style = {draw,line width=2pt},
brace/.style = {decorate,decoration={brace,amplitude=10pt},xshift=0pt,yshift=0pt,line width=1.5pt}
}
    \def\e{.2}
    \def\p{9}
    \def\q{16}
    \def\w{.1}
    \def\h{.3}

\node[right] at (0,-6.5) {\color{darkblue}\Large$\mathsf{B}_{\mathsf{right}}$};

\begin{scope}[xscale=-1,yscale=-1]
\foreach \x in {0,...,\q} {
    \draw [frame] (\x,0) -- (\x,\p);
}

\foreach \y in {0,...,\p} {
    \draw [frame] (0,\y) -- (\q+.4,\y);
}
\end{scope}

\begin{scope}[xscale=-1,yscale=-1]
    
    \foreach \x in {0,...,6} {
        \begin{scope}[biclique,shift={(\x+2,\p-\x-1)}]
        \draw[fill=gray!30] (0,0) -- (1,0) -- (1,1) -- (0,1) -- cycle;
        \end{scope}
    }
    
    \foreach \x in {0,1} {
        \begin{scope}[biclique,shift={(\p-2-\x,1-\x)}]
        \draw[fill=gray!30] (0,0) -- (1,0) -- (1,1) -- (0,1) -- cycle;
        \end{scope}
    }

    \foreach \x in {0,...,5} {
        \foreach \y in {1,2} {
             \begin{scope}[lightblue,biclique,shift={(\x-\y+2,\p-\x-1-\y)}]
                \draw[fill=lightblue!30] (0,0) -- (1,0) -- (1,1) -- (0,1) -- cycle;
            \end{scope}
        }
    }
    
    \begin{scope}[lightblue] 
        \begin{scope}[biclique,shift={(0,\p-1)}]
        \node[dot] at (0,0) {};
        \draw (0,0) -- (1-\e,0);
        \draw (0,-1+\e) -- (0,1-\e);
        \end{scope}
        
        \begin{scope}[biclique,shift={(1,\p)}]
        \node[dot] at (0,0) {};
        \draw (-1+\e,0) -- (1-\e,0);
        \draw (0,-1+\e) -- (0,0);
        \end{scope}
        
        %%%%% K_{1,3}
        \foreach \x in {1,3,5} {
            \begin{scope}[biclique,shift={(\x,0)},rotate=0]
            \node[dot] at (0,0) {};
            \draw (-1+\e,0) -- (1-\e,0);
            \draw (0,0) -- (0,1-\e);
            \end{scope}
        }
        
        \foreach \y in {1,3,5} {
            \begin{scope}[biclique,shift={(0,\y)},rotate=-90]
            \node[dot] at (0,0) {};
            \draw (-1+\e,0) -- (1-\e,0);
            \draw (0,0) -- (0,1-\e);
            \end{scope}
        }
        
        %%%%% K_{1,4}
        \foreach \x in {0,1} {
            \begin{scope}[biclique,shift={(1+\x,2-\x)}]
            \node[dot] at (0,0) {};
            \draw (-1+\e,0) -- (1-\e,0);
            \draw (0,-1+\e) -- (0,1-\e);
            \end{scope}
        }
        
        \foreach \x in {0,...,3} {
            \begin{scope}[biclique,shift={(1+\x,4-\x)}]
            \node[dot] at (0,0) {};
            \draw (-1+\e,0) -- (1-\e,0);
            \draw (0,-1+\e) -- (0,1-\e);
            \end{scope}
        }
        
    \begin{scope}[biclique,darkblue,shift={(0,6)}]
    \draw[fill=darkblue!30] (0,0) -- (1,0) -- (1,1) -- (0,1) -- cycle;
    \end{scope}
\end{scope}
    
    \begin{scope}[shift={(\p,\p)},xscale=-1,yscale=-1]
        
    \begin{scope}[biclique,shift={(0,\p-1)}]
    \node[dot] at (0,0) {};
    \draw (0,0) -- (1-\e,0);
    \draw (0,-1+\e) -- (0,1-\e);
    \end{scope}
    
    \begin{scope}[biclique,shift={(1,\p)}]
    \node[dot] at (0,0) {};
    \draw (-1+\e,0) -- (1-\e,0);
    \draw (0,-1+\e) -- (0,0);
    \end{scope}
    
    %%%%% K_{1,3}

    \foreach \x in {1,3,5} {
        \begin{scope}[biclique,shift={(\x,0)},rotate=0]
        \node[dot] at (0,0) {};
        \draw (-1+\e,0) -- (1-\e,0);
        \draw (0,0) -- (0,1-\e);
        \end{scope}
    }
    
    \foreach \y in {1,3,5} {
        \begin{scope}[biclique,shift={(0,\y)},rotate=-90]
        \node[dot] at (0,0) {};
        \draw (-1+\e,0) -- (1-\e,0);
        \draw (0,0) -- (0,1-\e);
        \end{scope}
    }
    
    %%%%% K_{1,4}
    \foreach \x in {0,1} {
        \begin{scope}[biclique,shift={(1+\x,2-\x)}]
        \node[dot] at (0,0) {};
        \draw (-1+\e,0) -- (1-\e,0);
        \draw (0,-1+\e) -- (0,1-\e);
        \end{scope}
    }
    
    \foreach \x in {0,...,3} {
        \begin{scope}[biclique,shift={(1+\x,4-\x)}]
        \node[dot] at (0,0) {};
        \draw (-1+\e,0) -- (1-\e,0);
        \draw (0,-1+\e) -- (0,1-\e);
        \end{scope}
    }
    
    \end{scope}

\end{scope}

\begin{scope}[shift={(-18,0)},yscale=-1]

    \foreach \x in {0,...,6} {
        \begin{scope}[biclique,shift={(\x+2,\p-\x-1)}]
        \draw[fill=gray!30] (0,0) -- (1,0) -- (1,1) -- (0,1) -- cycle;
        \end{scope}
    }
    
    \foreach \x in {0,1} {
        \begin{scope}[biclique,shift={(\p-2-\x,1-\x)}]
        \draw[fill=gray!30] (0,0) -- (1,0) -- (1,1) -- (0,1) -- cycle;
        \end{scope}
    }
    
    \begin{scope}[shift={(\p,\p)},xscale=-1,yscale=-1]
        
    \begin{scope}[biclique,shift={(0,\p-1)}]
    \node[dot] at (0,0) {};
    \draw (0,0) -- (1-\e,0);
    \draw (0,-1+\e) -- (0,1-\e);
    \end{scope}
    
    \begin{scope}[biclique,shift={(1,\p)}]
    \node[dot] at (0,0) {};
    \draw (-1+\e,0) -- (1-\e,0);
    \draw (0,-1+\e) -- (0,0);
    \end{scope}
    
    %%%%% K_{1,3}

    \foreach \x in {1,3,5} {
        \begin{scope}[biclique,shift={(\x,0)},rotate=0]
        \node[dot] at (0,0) {};
        \draw (-1+\e,0) -- (1-\e,0);
        \draw (0,0) -- (0,1-\e);
        \end{scope}
    }
    
    \foreach \y in {1,3,5} {
        \begin{scope}[biclique,shift={(0,\y)},rotate=-90]
        \node[dot] at (0,0) {};
        \draw (-1+\e,0) -- (1-\e,0);
        \draw (0,0) -- (0,1-\e);
        \end{scope}
    }
    
    %%%%% K_{1,4}
    \foreach \x in {0,1} {
        \begin{scope}[biclique,shift={(1+\x,2-\x)}]
        \node[dot] at (0,0) {};
        \draw (-1+\e,0) -- (1-\e,0);
        \draw (0,-1+\e) -- (0,1-\e);
        \end{scope}
    }
    
    \foreach \x in {0,...,3} {
        \begin{scope}[biclique,shift={(1+\x,4-\x)}]
        \node[dot] at (0,0) {};
        \draw (-1+\e,0) -- (1-\e,0);
        \draw (0,-1+\e) -- (0,1-\e);
        \end{scope}
    }
    
    \end{scope}

\end{scope}

\end{tikzpicture}}
    \caption{}
    \label{fig:grid-10_0_mirror}
\end{subfigure}
\caption{}
\label{fig:ends}
\end{figure}
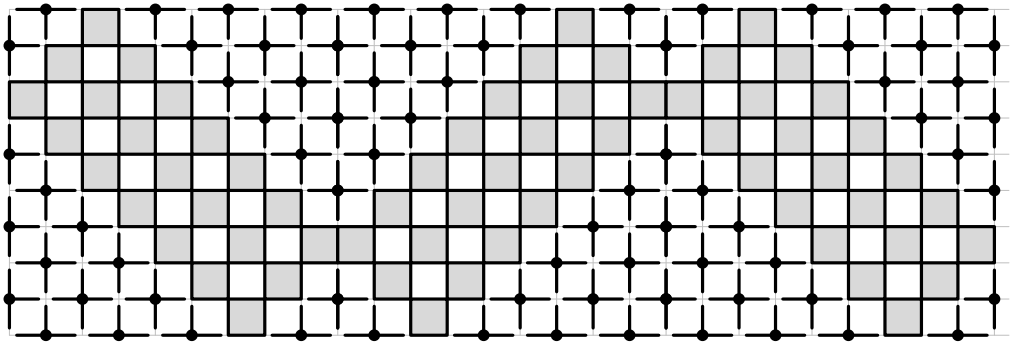
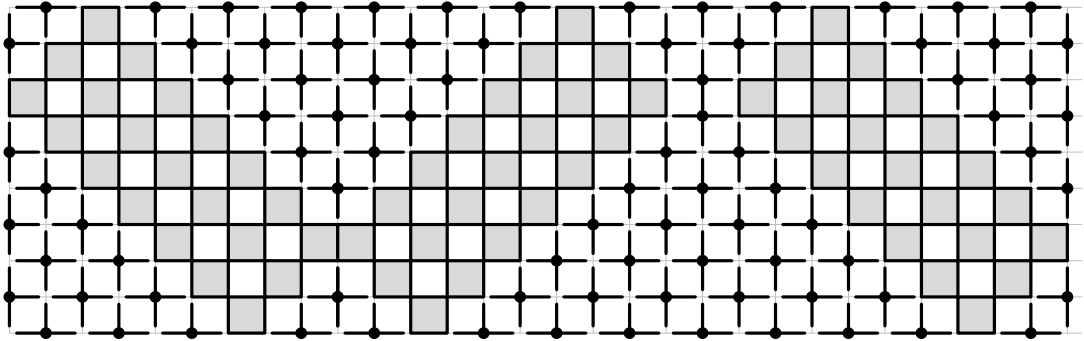
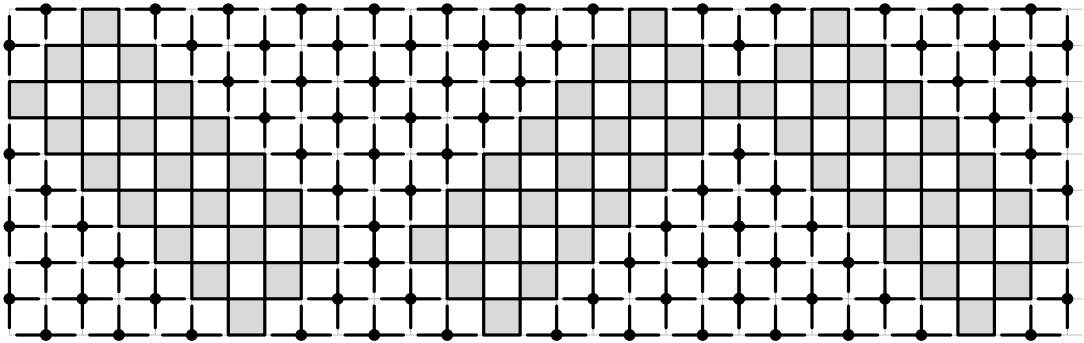
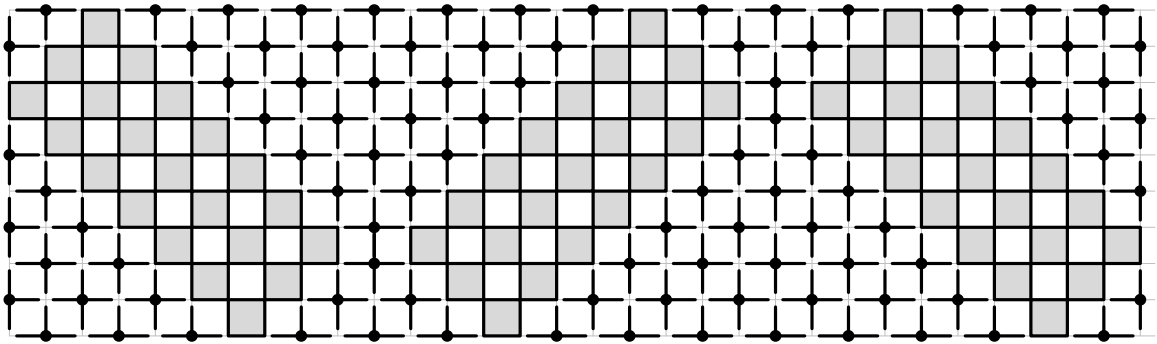
\begin{figure}[pt]\centering
\begin{subfigure}{1\textwidth}
	\scalebox{.6}{\input{grid-10_2.tex}}
    \caption{}
    \label{fig:grid-10_2}
\end{subfigure}\\[3mm]
\begin{subfigure}{1\textwidth}
	\scalebox{.6}[-.6]{\input{grid-10_1.tex}}
    \caption{}
    \label{fig:grid-10_1}
\end{subfigure}\\[3mm]
\begin{subfigure}{1\textwidth}
	\scalebox{.6}{\input{grid-10_4.tex}}
    \caption{}
    \label{fig:grid-10_4}
\end{subfigure}\\[3mm]
\begin{subfigure}{1\textwidth}
	\scalebox{.6}[-.6]{\input{grid-10_3.tex}}
    \caption{}
    \label{fig:grid-10_3}
\end{subfigure}
\caption{}
\label{fig:grid-10}
\end{figure}
This completes the entire proof.
\end{proof}

\subsection*{Acknowledgements} 
We thank Denis Cornaz for suggesting this problem and for carefully reading an early draft of this paper. We also thank Matthias Walter for help in computing $\bc(G_{p,q})$ for small values of $p$ and $q$ via SCIP~\cite{SCIP6}.  This project is supported by ERC grant \emph{FOREFRONT} (grant agreement no. 615640) funded by the European Research Council under the EU's 7th Framework Programme (FP7/2007-2013).

\bibliographystyle{amsplain}
\bibliography{main.bib}
\end{document}